\documentclass[12pt,twoside,leqno,a4paper]{amsart}
\usepackage{amsmath}
\usepackage{amssymb}
\usepackage[english]{babel}
\usepackage{color}

\theoremstyle{plain}
\newtheorem{thm}{Theorem}[section]
\newtheorem{lem}[thm]{Lemma}

\newtheorem{prop}[thm]{Proposition}
\newtheorem{cor}[thm]{Corollary}

\newtheorem*{Chernoff:inequality}{Chernoff's inequality}

\newcommand{\C}{{\mathbb C}}

\newcommand{\N}{{\mathbb N}}

\newcommand\pla{^{p,\ell}_{\alpha}}

\newcommand\dla{^{2,\ell}_{\alpha}}

\newcommand\la{^{\ell}_{\alpha}}

\newcommand{ \hvla}{\mathfrak{h}^{\ell}_{\varphi,\alpha}}
\newcommand{ \hbla}{\mathfrak{h}^{\ell}_{b,\alpha}}

\renewcommand{\Re}{\operatorname{Re}}

\DeclareMathOperator*{\esssup}{ess\,sup}

\title[Small Hankel operators on generalized Fock spaces]{Small Hankel operators on generalized Fock spaces}

\author{Carme Cascante}
\address{C. Cascante: Departament de Matem\`atiques i
    Inform\`atica, Universitat  de Barcelona,
     Gran Via 585, 08071 Barcelona, Spain.
 BGSMath.}
\email{cascante@ub.edu}
\author{Joan F\`abrega}
\address{J. F\`abrega: Departament de Matem\`atiques i
    Inform\`atica, Universitat  de Barcelona,
     Gran Via 585, 08071 Barcelona, Spain}
\email{joan$_{-}$fabrega@ub.edu}
\author{Daniel Pascuas}
\address{D. Pascuas: Departament de Matem\`atiques i
    Inform\`atica, Universitat  de Barcelona,
     Gran Via 585, 08071 Barcelona, Spain}
\email{daniel$_{-}$pascuas@ub.edu}
\author{Jos\'e A. Pel\'aez}
\address{J. A. Pel\'aez: Departamento de An\'alisis
      Matem\'atico, Facultad de Ciencias,
       29071 M\'alaga, Spain}
 \email{japelaez@uma.es}


\date{\today}

\thanks{The research of the first three authors was supported in part by
        Ministerio de Econom\'{\i}a y Competitividad, Spain,   projects MTM2014-51834-P and MTM2015-69323-REDT, and Generalitat de Catalunya, project
2014SGR289.\newline
    The research of the fourth author was supported in part by
    Ministerio de Econom\'{\i}a y Competitividad, Spain,   projects
        MTM2014-52865-P, and MTM2015-69323-REDT; La Junta de Andaluc{\'i}a,
        project FQM210}
\begin{document}

\begin{abstract}
    We consider Fock spaces $F\pla$ of entire functions on $\C$ associated to the weights $e^{-\alpha |z|^{2\ell}}$, where $\alpha>0$ and $\ell$ is a positive integer. We compute explicitly the corresponding Bergman kernel associated to $F\dla$  and, using an adequate factorization of this kernel, we characterize the boundedness and the compactness of the small Hankel operator $\hbla$ on $F\pla$. Moreover, we also determine when  $\hbla$  is a Hilbert-Schmidt operator on $F\dla$.
\end{abstract}
\maketitle

\section{Introduction}

Let  $\alpha>0$. For $1\le p<\infty$, we denote by  $L\pla$ the space of all measurable functions $f$ on $\C$ such that
\[
\|f\|_{L\pla}^p:=\int_\C \left|f(z) e^{-\alpha|z|^{2\ell}/2}\right|^p \,d\nu(z) < \infty,
\]
where   $d\nu$ denotes the Lebesgue measure on $\C$.

For $p=\infty$, $L^{\infty,\ell}_{\alpha}$ denotes the space of all measurable functions $f$ on $\C$ such that
\[
\|f\|_{L^{\infty,\ell}_{\alpha}}:= \esssup_{z\in\C} |f(z) e^{-\alpha|z|^{2\ell}/2}| < \infty.
\]
Note that $L\pla = L^p(\C, e^{-p\alpha |z|^{2\ell}/2}\, d\nu)$, $1\le p < \infty$.
So, for $1\le p\le\infty$,
$(L\pla, \|\,\cdot\,\|_{L\pla})$ is a Banach space
 and $L\dla$ is a Hilbert space with
the inner product
\[
\langle f,\,g\rangle\la:=\int_{\C}f(z)\overline{g(z)}\, e^{-\alpha |z|^{2\ell}}\,d\nu(z)
\quad(f,g\in L\dla).
\]

 The generalized Fock spaces are defined to be
\[
F\pla := H(\C) \cap L\pla \qquad (1\le p \leq \infty),
\]
where $H(\C)$ denotes the space of entire functions. It is well known that the space of holomorphic polynomials is dense in $F\pla$ for $p<\infty$.

If $p=2$, $F\dla$ is  a Hilbert space. We will denote by $P\la$ the orthogonal projection from $L\dla$ to $F\dla$, which is an integral operator whose kernel is  $K\la$, the Bergman reproducing kernel for $F\dla$.

It is also convenient to consider the little Fock space
\[
\mathfrak{f}^{\infty,\ell}_{\alpha}:=\bigl\{\,f\in H(\C)\,:\,\lim_{|z|\to\infty}|f(z)|e^{-\alpha|z|^{2\ell}/2}=0\,\bigr\},
\]
which is the closure of the space of all holomorphic polynomials in $F^{\infty,\ell}_{\alpha}$.

Recall that for $\ell=1$  one obtains the classical
Fock spaces $F^p_{\alpha}$ and $\mathfrak{f}^\infty_{\alpha}$.

The main goal of this paper is to characterize the boundedness and the compactness of the small Hankel operators
$$
\hbla(f):=\overline{P\la(b\overline{f})}
$$
on  $F\pla$ for the whole range $1\le p<\infty$.

For the classical case $\ell=1$ and $p=2$, it is  well known  that, if $b\in
{F}^2_{\alpha}$, the small Hankel operator
$
\mathfrak{h}^1_{b,\alpha}(f):=\overline{P^{1}_\alpha(b\overline{f})}
$
is
bounded (compact)  from $F^2_\alpha$ to
$\overline{F^2_\alpha}$ if and only if $b\in
F^\infty_{\alpha/2}$  ($b\in
\mathfrak{f}^\infty_{\alpha/2}$).
Moreover, $\mathfrak{h}^1_{b,\alpha}$ is a Hilbert-Schmidt operator if and only if $b\in F^2_{\alpha/2}$
(see~{\cite{janson-peetre-rochberg}
    and~{\cite{Zhu2012}}}).

Up to our knowledge, there are not known results on small Hankel operators for  $\ell>1$.
This is not the case for the big Hankel operator  $\mathcal{H}_{\overline{b}}(f):= \overline{b}f- P^1_\alpha( \overline{b}f)$.
 In  \cite{bommier-youssfi} (see also  \cite{constantin--ortega-cerda}) the authors  prove that   $\mathcal{H}_{\overline{b}}$
 is  a bounded operator on $F^{2,\ell}_\alpha$ if and only if $b'(z)(1+|z|)^{1-\ell}\in L^\infty$, that is, $b$ is a polynomial of degree at most $\ell$.
 It is also worth mentioning \cite{SeiYouJGA2011}, where are described  the bounded, compact and Schatten class big Hankel  operators on Hilbert Fock spaces induced by radial
 rapidly decreasing weights.

 Observe that $(1+|z|)^{1-\ell}\simeq (\Delta |z|^{2\ell})^{-1/2}$, $|z|\ge 1$. It is well known that in the general theory of Fock spaces $F^{p}_\phi$, the Laplacian of the subharmonic weight $\phi$ plays an important role (see, for instance, the recent papers \cite{constantin-pelaez} and \cite{oliver-pascuas}  and the references therein).
A natural question from those observations, which will be solved by the main results of this paper,
 is whether or not the
boundedness  of $\hbla$ on $F\dla$ is described
by conditions on $b$ involving  $\Delta |z|^{2\ell}$.

In order to introduce a natural space of symbols to study the small Hankel operator acting on $F\pla$, notice that
if  $\hbla$  is bounded on $F\pla$,
for some $b\in H(\C)$, then $b=\overline {\hbla(1)}\in F\pla$.
For the classical case, we have
$F^p_\alpha\subset F^\infty_\alpha$, if $1\le p<  \infty$.
Those considerations suggest that the appropriate space of symbols in this classical setting is
 $F^\infty_\alpha$.
 When $\ell>1$ the
inclusion $F\pla\subset F^{\infty,\ell}_\alpha$ is  no longer true (see for instance  \cite[Corollary~2]{ConsPelJGA2016}).
Instead, for any function $b$ in $F\pla$ the pointwise estimate
$$
|b(z)|\lesssim \|b\|_{F\pla}
(1+|z|)^{(2\ell-2)/p}e^{\alpha|z|^{2\ell}/2}
$$
holds
(see~{\cite[Lemma 19(a)]{marco-massaneda-ortega}}).
Hence, in the general setting we consider  the space of holomorphic symbols given by
\begin{equation*}
H^{\infty,\ell}_\alpha:=\left\{b\in H(\C):\,|b(z)|=O\bigl((1+|z|)^{2\ell-2}e^{\alpha|z|^{2\ell}/2}\bigr)\right\}.
\end{equation*}

Assuming $b\in H^{\infty,\ell}_\alpha$, the
operator $\hbla$  is well defined on the space $E$ of entire functions of order $\ell$ and finite type, that is,
\begin{equation}\label{eqn:E}
E:=\{f\in H(\C):\,|f(z)|=O(e^{\beta|z|^\ell}),\quad\text{for some }\,\, \beta>0\}.
\end{equation}

Since $E$ contains the space of the holomorphic polynomials,  $E$ is dense in  $\mathfrak{f}^{\infty,\ell}_\alpha$ and in $F\pla$, for any $p<\infty$.

Our main results are the following.

\begin{thm}\label{thm:main1}
Let $\alpha>0$, $\ell\in \N$, $b\in H^{\infty,\ell}_\alpha$ and $1\le p<\infty$. Then
  $\hbla$ is a bounded operator  from  $F\pla$  to $\overline{F\pla}$  if and only if
 $b\in F^{\infty,\ell}_{\alpha/2}$.
 In such  case,
$
\|\hbla\|_{F\pla}\simeq
\|b\|_{F^{\infty,\ell}_{\alpha/2}}.
$

Analogously, $\hbla$ is a bounded operator  from
 $\mathfrak{f}^{\infty,\ell}_\alpha$ to $\overline{\mathfrak{f}^{\infty,\ell}_\alpha}$
  if and only if
  $b\in F^{\infty,\ell}_{\alpha/2}$ and
 $ \|\hbla\|_{\mathfrak{f}^{\infty,\ell}_\alpha}\simeq\|b\|_{F^{\infty,\ell}_{\alpha/2}}.$
\end{thm}

Here and throughout the paper $\|\hbla\|_{F\pla}$ denotes the norm of $\hbla$ as an operator from $F\pla$ to $\overline{F\pla}$.


Since the boundedness of  small Hankel operators is equivalent
to the boundedness of the corresponding Hankel forms,
 as an application of Theorem \ref{thm:main1} we obtain:

\begin{thm}\label{thm:main3}
Let $1< p<\infty$, $\ell\in\N$, $\alpha>0$ and $b\in H^{\infty,\ell}_\alpha$.
\begin{enumerate}
\item \label{item:main31} Let $\Lambda^\ell_{b,\alpha}$ be the Hankel bilinear form
defined by
$$
\Lambda^\ell_{b,\alpha}(f,g):=\langle fg,b\rangle\la
\qquad(f,g\in E).
$$
Then, $\Lambda^\ell_{b,\alpha}$ extends to a bounded bilinear form either on
$F^{p,\ell}_\alpha\times F^{p',\ell}_\alpha$ or on $F^{1,\ell}_\alpha\times \mathfrak{f}^{\infty,\ell}_\alpha$
if and only if $b\in F^{\infty,\ell}_{\alpha/2}$.
    \item \label{item:main32}
   The space $F^{\infty,\ell}_{\alpha/2}$ coincides with $P\la(L^\infty)$ and also  with the dual of $F^{1,\ell}_{2\alpha}$ with respect to the pairing
    $\langle \cdot,\cdot\rangle\la$.
    \item \label{item:main33}
    $F^{p,\ell}_\alpha\odot F^{p',\ell}_\alpha
        =F^{1,\ell}_\alpha\odot \mathfrak{f}^{\infty,\ell}_\alpha=F^{1,\ell}_\alpha\odot F^{\infty,\ell}_\alpha= F^{1,\ell}_{2\alpha}$.
\end{enumerate}
\end{thm}

Here and troughout the paper, $p'$ denotes the conjugate exponent of $p$.

We recall that the weak product $F^{p,\ell}_\alpha\odot
F^{p',\ell}_\alpha$ consists of all entire functions
$h=\sum_{j=1}^\infty f_j g_j$, $f_j\in F^{p,\ell}_\alpha$ and $g_j\in
F^{p',\ell}_\alpha$, such that
$$
\|h\|_{F^{p,\ell}_\alpha\odot F^{p',\ell}_\alpha}
:=\inf\left\{\sum_{j=1}^\infty \|f_j\|_{F^{p,\ell}_\alpha}\|g_j\|_{F^{p',\ell}_\alpha}: h=\sum_{j=1}^\infty f_jg_j\right\}<\infty.
$$

\begin{thm}\label{thm:main2}
    Let $\alpha>0$, $\ell\in \N$, $b\in H^{\infty,\ell}_\alpha$ and $1\le p<\infty$.
     Then $\hbla$  is compact from $F\pla$ to $\overline{F\pla}$  if and only if
    $b\in \mathfrak{f}^{\infty,\ell}_{\alpha/2}$.

    Similarly, $\hbla$  is compact from $\mathfrak{f}^{\infty,\ell}_{\alpha}$ to $\overline{\mathfrak{f}^{\infty,\ell}_{\alpha}}$  if and only if
    $b\in \mathfrak{f}^{\infty,\ell}_{\alpha/2}$.
\end{thm}

As far as we know, the techniques that have been
used to prove  characterizations of the  boundedness and the compactness of   small
Hankel operators
on the classical Fock spaces $F^p_\alpha=F^{p,1}_\alpha$ (see \cite{janson-peetre-rochberg,Zhu2012})
are strongly  based on the
fact that the Bergman reproducing kernel of $F^2_\alpha$ is given
by the neat expression
$
K^1_\alpha(z,w)=\frac{\alpha}{\pi}e^{\alpha{z\overline{w}}},
$
which permits to factorize the kernel as
 \begin{equation}\label{eqn:factorization1}
 K^1_\alpha(z,w)
=\frac{\pi}{\alpha}K^1_\alpha(z/2,w)K^1_\alpha(z/2,w).
\end{equation}
Thus, the proof  is quite easy  since
the integral operator with kernel  $K^1_\alpha(z/2,\cdot)$ maps the function $f$ in the Fock space to the function $f(\cdot/2)$.
However, the
general situation on $F^{2,\ell}_\alpha$, $\ell>1$,
is much more involved because of the lack  of such a simple expression for $K\la$.
In this general case we use the factorization
\begin{equation}\label{eqn:factorization2}
K^\ell_\alpha(w,z)   = G_{\alpha,0}(w,z)G_{\alpha,1}(w,z),
\end{equation}
where
\[
G_{\alpha,0}(w,z):= e^{\frac{\alpha}2(w\overline z)^\ell}\quad\text{and}\quad G_{\alpha,1}(w,z):= e^{-\frac{\alpha}2(w\overline z)^\ell}K^\ell_\alpha(w,z),
\]
which for $\ell=1$ is just \eqref{eqn:factorization1}.

Note that \eqref{eqn:factorization2}, which is given  in terms of analytic functions, is possible because $\ell$ is a positive integer. For other values of $\ell$ it is not clear how to choose a suitable decomposition.

Finally, we characterize the membership of $\hbla$  in the class $\mathcal{S}_2(F^{2,\ell}_{\alpha})$ of  Hilbert-Schmidt operators from
$F^{2,\ell}_{\alpha}$ to $\overline{F^{2,\ell}_{\alpha}}$.

For $\ell=1$, $\hbla\in \mathcal{S}_2(F^{2}_{\alpha})$ if and only if $b\in F^{2}_{\alpha/2}$ (see~{\cite{janson-peetre-rochberg}
    or~{\cite{Zhu2012}}}). For $\ell>1$ the characterization is given in terms of the space  $F^{2,\ell}_{\alpha,\Delta}$ of all  functions $f\in H(\C)$ such that
$$
\|f\|^2_{F^{2,\ell}_{\alpha,\Delta}}:=\int_{\C}|f(z) e^{-\frac{\alpha}{2}|z|^{2\ell}}|^2\, (1+|z|)^{2(\ell-1)}\,d\nu(z)<\infty.
$$

\begin{thm}\label{thm:main4}
Let $\alpha>0$, $\ell\in \N$  and $b\in H^{\infty,\ell}_\alpha$.
Then, $\hbla \in \mathcal{S}_2(F^{2,\ell}_{\alpha})$ if and only if $b\in F^{2,\ell}_{\alpha/2,\Delta}$.
 Moreover,
$$\|\mathfrak{h}^{\ell}_{b,\alpha}\|_{\mathcal{S}_2(F^{2,\ell}_{\alpha})}\simeq
\|b\|_{F^{2,\ell}_{\alpha/2,\Delta}}.$$
\end{thm}

Observe that, while the descriptions of the boundedness and compactness
of the small Hankel operators on $F\pla$ obtained in Theorems
\ref{thm:main1} and \ref{thm:main2} do not depend on the Laplacian
of $|z|^{2\ell}$, this is not the case for
Hilbert-Schmidt operators.
Taking into account our results, it seems natural to
conjecture analogous ones for weighted  Fock spaces induced by
weights $e^{-\phi}$, where $\phi$ is a subharmonic function such
that $\Delta\phi$ is a doubling measure.

The paper is organized as follows. In Section \ref{sect:prelim}  we state some useful properties of the Bergman projection, as well as  the main properties of the spaces $F\pla$ and of the small Hankel operator.
In Section \ref{sect:ProofThm1} we prove Theorem
\ref{thm:main1}.
In Sections \ref{sect:ProofThm3} and  \ref{sect:ProofThm2} we give the  proof of Theorems \ref{thm:main3} and \ref{thm:main2}, respectively.
Finally, in Section \ref{sect:ProofThm4} we provide a proof of Theorem \ref{thm:main4}, which follows  from the definition of the Hilbert-Schmidt norm.

\subsection{Notations}
Throughout the paper, $\N$ denotes the set of all positive integers.
We denote by $p'$ the conjugate exponent of $p$.
The letter $C$ will denote a positive constant, which may vary from place to
 place. The notation $A\lesssim B$ means that there exists a constant $C>0$, which does not depend on the involved variables,
 such that $A\le C\, B$. We write $A\simeq B$ when $A\lesssim B$ and $B\lesssim A$. We will also say that $\hbla$ is bounded (compact) on $F\pla$ if it is bounded (compact) from $F\pla$ to $\overline{F\pla}$. We denote the norm of this operator by $\|\hbla\|_{F\pla}$. The same notations will be used replacing  $F\pla$ by $\mathfrak{f}^{\infty,\ell}_\alpha$.

\section{Preliminaries}\label{sect:prelim}

\subsection{Properties of the  Fock spaces
$F^{p,\ell}_{\alpha}$}\quad\par

We begin the subsection recalling some useful embeddings of the generalized Fock spaces.

\begin{lem}\label{lem:propertiesF}
    Let $1\le p,q\le \infty$.  If $0<\alpha<\beta<\gamma<\delta$ then we have the  embeddings
    \[
    F^{\infty,\ell}_{\alpha}\hookrightarrow F^{p,\ell}_{\beta}\hookrightarrow  H^{\infty,\ell}_{\beta}\hookrightarrow \mathfrak{f}^{\infty,\ell}_{\gamma} \hookrightarrow F^{q,\ell}_{\delta}.
    \]
\end{lem}

\begin{proof}
As we said in the introduction, the embedding
$F^{p,\ell}_{\beta}\hookrightarrow H^{\infty,\ell}_{\beta}$ is proved in
\cite[Lemma 19(a)]{marco-massaneda-ortega}.
The rest follows directly.
\end{proof}

Since our weights $\alpha|z|^{2\ell}/2$ are radial, the dilations
$z\mapsto\lambda z$, $\lambda>0$, act isometrically on our spaces
$L^{p,\ell}_{\alpha}$ and $F^{p,\ell}_{\alpha}$, as it is stated in the following
proposition.

\begin{prop}\label{prop:dilations:act:isometrically}
    Let $1\le p\le \infty$, $\alpha,\lambda>0$ and $\ell\in\N$. For any function $f$ on $\C$ we define
\begin{equation}\label{eqn:dilations:act:isometrically}
    \Phi^{\ell}_{\lambda}f(z):=f(\lambda^{1/(2\ell)}z)
    \qquad(z\in\C).
\end{equation}
Then $\Phi^{p,\ell}_{\lambda}:=
\lambda^{1/(p\ell)}\,\Phi^{\ell}_{\lambda}$ is
a linear isometry from $L\pla$ onto
$L^{p,\ell}_{\lambda\alpha}$ such that
$\Phi^{p,\ell}_{\lambda}(F\pla)=F^{p,\ell}_{\lambda\alpha}$ and
$\Phi^{p,\ell}_{\lambda}(\mathfrak{f}^{\infty,\ell}_{\alpha})=\mathfrak{f}^{\infty,\ell}_{\lambda\alpha}$.
In particular,
\[
\langle\Phi^{2,\ell}_{\lambda}f,\,
\Phi^{2,\ell}_{\lambda}g\rangle^{\ell}_{\lambda\alpha}=
\langle f,\,g\rangle^{\ell}_{\alpha}
\qquad(f,g\in L^{2,\ell}_{\alpha}).
\]
\end{prop}

\begin{proof}
The first assertion follows by making the change of variable
$w=\lambda^{1/(2\ell)}z$. The second assertion is a direct consequence of the
first one for $p=2$.
\end{proof}

\subsection{The Bergman kernel}\label{section:bergmankernel}\quad\par

It is well-known that
$F\dla$ with the inner product $\langle \cdot,\cdot\rangle\la$
is a Hilbert space such that the pointwise evaluation $f \mapsto f(z)$ is a bounded linear functional on $F\dla$,
for any $z\in\C$.
 Thus $F\dla$ is a reproducing kernel Hilbert space, that is, for any $z\in\C$
 there exists a unique function $K^{\ell}_{\alpha,z}$ in $F\dla$ such that
$f(z) = \langle f,\,K^{\ell}_{\alpha,z}\rangle\la$, for every $f\in F\dla$.
  The {\em Bergman kernel} for $F\dla$ is the function
\[
K\la(z,w):=K^{\ell}_{\alpha,w}(z)=\overline{K^{\ell}_{\alpha,z}(w)}\qquad(z,w\in\C).
\]

The following result is well known (see for instance
\cite{bommier-englis-youssfi}).

\begin{prop}\label{prop:kernel}
Let $\alpha>0$ and $\ell\in\N$.
Then the sequence of monomials $\{z^m\}_{m\ge0}$ is an
orthogonal basis of $F\dla$ and
\begin{equation*}
\|z^m\|_{F\dla}^2=\frac{\pi}{\ell\alpha^{(m+1)/\ell}}\Gamma\left(\frac{m+1}{\ell}\right).
\end{equation*}
Therefore, the sequence
 \[\{e_m\}_{m\ge0}:=
 \left\{\frac{z^m}{\|z^m\|_{F\dla}}\right\}=
 \left\{\sqrt{\frac{\ell}{\pi}\,\frac{\alpha^{(m+1)/\ell}}{\Gamma\left(\frac{m+1}{\ell}\right)}}\,z^m\right\}_{m\ge0}
 \]
is an orthonormal basis of $F\dla$ and
the Bergman kernel for $F\dla$ admits the representation
\begin{equation}\label{eqn:kernel}
K\la(z,w)
=\sum_{m=0}^\infty\frac{z^m\,\overline{w}^m}{\|w^m\|_{F\dla}\,\|z^m\|_{F\dla}}
=\frac{\ell\alpha^{1/\ell}} {\pi}\sum_{m=0}^\infty\frac{\alpha^{m/\ell} z^m\,\overline{w}^m}
{\Gamma\left(\frac{m+1}{\ell}\right)}.
\end{equation}

In particular,
\begin{equation}\label{eqn:kernel:change:parameter}
K^{\ell}_{\alpha}(z,w)=\alpha^{1/\ell}\,
K^{\ell}_1(\alpha^{1/(2\ell)}z,\,\alpha^{1/(2\ell)}w).
\end{equation}
 \end{prop}

Formula \eqref{eqn:kernel}  shows that the Bergman kernel can be written in terms of the Mittag-Leffler functions. Namely,
\begin{equation}\label{eqn:KH}
K^{\ell}_{\alpha}(z,w) =H\la(z\overline w)=
\frac{\ell \alpha^{1/\ell}}{\pi} E_{\frac 1\ell,\frac 1\ell}(\alpha^{1/\ell}z\overline w) \quad(z,w\in\C),
\end{equation}
where
\[
E_{\frac 1\ell,\frac 1\ell}(\lambda)
=\sum_{k=0}^\infty
\frac{\lambda^k}{\Gamma\bigl(\frac{k+1}\ell\bigr)}
\quad(\lambda\in\C).
\]

It is known that the Mittag-Leffler function
$E_{\frac 1\ell,\frac 1\ell}(\lambda)$ satisfies the following asymptotic expansion as $|\lambda|\to\infty$ (see \cite[Chapter XVIII ]{Bateman-Erdelyi}):

\begin{equation}\label{eqn:asimpE}
E_{\frac 1\ell,\frac 1\ell}(\lambda)=
\begin{cases}
\ell\lambda^{\ell-1}e^{\lambda^\ell}+O(\lambda^{-1}),
&\text{if}\,\,\,|\arg(\lambda)|\le \frac{\pi}{2\ell},\\
O(\lambda^{-1}),
&\text{if}\,\,\,|\arg(\lambda)|> \frac{\pi}{2\ell}.
\end{cases}
\end{equation}
Here $\arg(\lambda)$ denotes the principal branch
of the argument of $\lambda$, that is,
 $-\pi<\arg(\lambda)\le\pi$.

It is clear that \eqref{eqn:asimpE} implies the
 following pointwise estimate of the Bergman kernel.

\begin{prop}\label{prop:pointwise}
		\[
	|K^\ell_\alpha(z,w)|\lesssim
	(1+|z\overline w|)^{\ell-1}
	\left(e^{\alpha\Re((z\overline w)^\ell)}+1\right) \quad(z,w\in\C).
	\]
\end{prop}

Observe that \eqref{eqn:asimpE} also gives
pointwise estimates of $K^\ell_\alpha$ for $\ell$
not necessarily integer.
However, in this non-integer case to obtain a
 factorization of the Bergman kernel as in
 \eqref{eqn:factorization2}
seems more difficult.
Other estimates for more general radial weights are
 given in  \cite{SeiYouJGA2011}.

\subsection{The Bergman projection}
The orthogonal projection $P^\ell_{\alpha}$ from $L\dla$ onto
$F\dla$ admits the integral representation
\[
P^\ell_{\alpha}f(z):=
\int_{\C}K\la(z,w)\,f(w)\,e^{-\alpha|w|^{2\ell}}\,d\nu(w)
\quad(f\in L\dla,\,z\in\C).
\]
Note that if  $f\in L^{p,\ell}_\beta$, $1\le p<\infty$, $0<\beta<2\alpha$, then $P^\ell_{\alpha}f$ is well defined,
 that is, for any $z\in\C$, the function
\[
F_z(w)=K^\ell_{\alpha,z}(w)\,f(w)\,e^{-\alpha|w|^{2\ell}}
=K^\ell_{\alpha}(w,z)\,f(w)\,e^{-\alpha|w|^{2\ell}}
\]
is integrable on $\C$.
 Indeed, by Proposition \ref{prop:pointwise},
  $|F_z(w)|\le C_z G_z(w) H_z(w)$, where
$G_z(w):=|f(w)|\,e^{-\beta|w|^{2\ell}/2}$ and
$H_z(w):=(1+|w|)^{\ell-1}e^{\alpha |z|^\ell|w|^\ell}e^{-(\alpha-\beta/2)|w|^{2\ell}/2}$.
Since $G_z\in L^p(\C)$ and $H_z\in L^{p'}(\C)$,
H\"{o}lder's inequality gives that
$$
\int_\C |F_z(w)|d\nu(w)\le C_z \|f\|_{L^{p,\ell}_\beta}.
$$
Hence $u(f)=\int_\C F_z\,d\nu$ is a bounded linear form on $L^{p,\ell}_\beta$.
Since $u(P)=P(z)$, for every holomorphic polynomial $P$, and the holomorphic
polynomials are dense on $F^{p,\ell}_\beta$, it turns out that
    \begin{equation}\label{eqn:reproducing:property}
    b(z)=\int_\C K^\ell_{\alpha} (z,w)\,b(w)\,e^{-\alpha|w|^\ell} d\nu(w),
    \end{equation}
    for any $b\in F^{p,\ell}_{\beta}$.
 In particular,  since by Lemma
    \ref{lem:propertiesF}, $H^{\infty,\ell}_{\alpha}\subset F^{p,\ell}_{\beta}$, for
    $\beta>\alpha$,  \eqref{eqn:reproducing:property} also holds for any
    $b\in H^{\infty,\ell}_{\alpha}$.

\begin{prop} \label{prop:Interp-dual}
For $\ell\ge 1$ and $\alpha>0$ we have:
\begin{enumerate}
  \item\label{item:Interp-dual1} If $1\le p\le\infty$, then $ P\la$ is a bounded projection from $
L\pla$ onto $F\pla$.
    \item\label{item:Interp-dual3} If $1\le p<\infty$, then $(F^{p,\ell}_\alpha)^*\equiv F^{p',\ell}_\alpha$,
    with respect to the pairing $ \langle\cdot,\cdot\rangle\la$.
     \item\label{item:Interp-dual4} $(\mathfrak{f}^{\infty,\ell}_{\alpha})^*\equiv F^{1,\ell}_\alpha$,
    with respect to the pairing $ \langle\cdot,\cdot\rangle\la$.
\end{enumerate}
\end{prop}

\begin{proof}
    The proof of the first two assertions can be found, for instance,
    in~{\cite[Theorem 13 and Corollary 14]{constantin-pelaez}} and~{\cite[Theorems 3.1 and 3.6]{oliver-pascuas}}, so we only have to prove the last one.

    First note that if $b\in F^{1,\ell}_\alpha$ then
    $\langle\cdot,b\rangle\la\in  (\mathfrak{f}^{\infty,\ell}_\alpha)^*$
    and
    $\|\langle\cdot,b\rangle\la\|_{(\mathfrak{f}^{\infty,\ell}_\alpha)^*}
    \lesssim\|b
    \|_{F^{1,\ell}_\alpha}$.

Conversely, given $u\in(\mathfrak{f}^{\infty,\ell}_\alpha)^*$, we are going to
prove that there is  $b\in F^{1,\ell}_\alpha$ such that
$u=\langle\cdot,b\rangle\la$ and $\|b\|_{F^{1,\ell}_\alpha}\lesssim\|u\|_{(\mathfrak{f}^{\infty,\ell}_\alpha)^*}$.
Pick $\alpha/2<\beta<\alpha$. Then, by Lemma~{\ref{lem:propertiesF}}, we have the embedding
$F^{2,\ell}_{\beta}\hookrightarrow \mathfrak{f}^{\infty,\ell}_\alpha$ and so the
restriction of $u$ to $F^{2,\ell}_{\beta}$ is a bounded linear form on this space.
It follows that there is $g\in F^{2,\ell}_{\beta}$ such that
$u(f)=\langle f,g\rangle^\ell_\beta$, for every $f\in E$.
Now Proposition~{\ref{prop:dilations:act:isometrically}}
 for $\lambda=\alpha/\beta$ shows that
$b:=\Phi^{2,\ell}_{\lambda^2}g\in F^{2,\ell}_{\alpha^2/\beta}$ satisfies
\[
u(f)=\langle f,g\rangle^\ell_\beta=
\langle\Phi^{2,\ell}_{\lambda}f,
\,\Phi^{2,\ell}_{\lambda}g\rangle^{\ell}_{\alpha} \stackrel{(*)}{=}
\langle f,\,b\rangle^{\ell}_{\alpha},
\quad\mbox{ for every $f\in E$.}
\]
(Note that $(*)$  holds because both functions $f$ and $g$ are entire.)
Thus it only remains to prove that  $\|b\|_{L^{1,\ell}_{\alpha}}
\lesssim\|u\|_{(\mathfrak{f}^{\infty,\ell}_\alpha)^*}$.
 Recall that, by duality,
\[
\|b\|_{L^{1,\ell}_\alpha}=\sup_{\substack{f\in C_c(\C)\\\|f\|_{L^\infty}=1}}
\left|\int_\C f(z)e^{-\frac{\alpha}2|z|^{2\ell}}\,\overline{b(z)}\,d\nu(z)\right|
=\sup_{\substack{ f\in C_c(\C)\\\|f\|_{L^\infty}=1}}
|\langle T_{\alpha}f,b\rangle\la|,
\]
where $T_{\alpha}f(z):=f(z)e^{\frac{\alpha}2|z|^{2\ell}}$.
Note that $b=P^{\ell}_{\alpha}(b)$, because $b\in F^{2,\ell}_{\alpha^2/\beta}$ and
 $\alpha^2/\beta<2\alpha$. Therefore, for any $f\in C_c(\C)$, we have that
\[
\langle T_{\alpha} f,b\rangle\la=
\langle T_{\alpha} f,P_{\alpha}^{\ell}b\rangle\la\stackrel{(1)}{=}
\langle P_{\alpha}^{\ell}(T_ {\alpha} f),b\rangle\la\stackrel{(2)}{=}
u(P_{\alpha}^{\ell}(T_ {\alpha} f)),
\]
where $(1)$ follows from Fubini's theorem and $(2)$ holds since
$P_{\alpha}^{\ell}(T_ {\alpha} f)\in E$. And hence
\[
|\langle T_{\alpha} f,b\rangle\la|\le \|u\|_{(\mathfrak{f}^{\infty,\ell}_\alpha)^*}
\|P^{\ell}_{\alpha}\|_{L^{\infty,\ell}_{\alpha}}
\|T_ {\alpha} f\|_{L^{\infty,\ell}_{\alpha}}=
\|u\|_{(\mathfrak{f}^{\infty,\ell}_\alpha)^*}
\|P^{\ell}_{\alpha}\|_{L^{\infty,\ell}_{\alpha}}
\|f\|_{L^{\infty}},
\]
which gives that $\|b\|_{L^{1,\ell}_\alpha}\lesssim \|u\|_{(\mathfrak{f}^{\infty,\ell}_\alpha)^*}$.
\end{proof}

The last result of this subsection states that the dilation operators
$\Phi^{\ell}_{\lambda}$, defined by~{\eqref{eqn:dilations:act:isometrically}},
``conmute'' with the Bergman projections.

\begin{prop}\label{prop:dilations:conmute:Bergman:projection}
Let $1\le p\le\infty$, $\ell\in\N$ and $\alpha,\beta,\lambda>0$ such that $\beta<2\alpha$. Then
\begin{equation*}
\Phi^{\ell}_{\lambda}(P^{\ell}_{\alpha}f)=
P^{\ell}_{\lambda\alpha}(\Phi^{\ell}_{\lambda}f)
\qquad(f\in L^{p,\ell}_{\beta}).
\end{equation*}
\end{prop}

\begin{proof}
Let $f\in L^{p,\ell}_{\beta}$. Then
\[
\Phi^{\ell}_{\lambda}(P^{\ell}_{\alpha}f)(z)=
\int_{\C}K^{\ell}_{\alpha}(\lambda^{1/(2\ell)}z,w) f(w)
e^{-\alpha|w|^{2\ell}}d\nu(w).
\]
By making the change of variable $w=\lambda^{1/(2\ell)}v$ and taking into
account that
\[
\lambda^{1/\ell}K\la(\lambda^{1/(2\ell)}z,\lambda^{1/(2\ell)}v)=
K^{\ell}_{\lambda\alpha}(z,v),
\]
which follows from~{\eqref{eqn:kernel}},  we conclude  that  $\Phi^{\ell}_{\lambda}(P^{\ell}_{\alpha}f)(z)=P^{\ell}_{\lambda\alpha}(\Phi^{\ell}_{\lambda}f)(z)$.
\end{proof}

\subsection{The small Hankel operator on $F\pla$, $1\le p<\infty$}\quad\par

The next lemma gives some properties of the  subspace of entire functions $E$ defined in \eqref{eqn:E}.
\begin{lem}\label{lem:spaceE}
The space $E$ satisfies the following properties:
\begin{enumerate}
  \item $E\cdot E\subset E$.
    \item $E\subset F^{1,\ell}_\alpha$, for any $\alpha>0$
    \item $E$ contains the space of all the holomorphic polynomials.
 \item $E$ contains the space $Span\{ K^{\ell}_{\alpha,z}:z\in\C\}$, i.e. the set of
 finite linear combinations of functions
$K^{\ell}_{\alpha,z}$.
 \item $E$ is dense in $\mathfrak{f}^{\infty,\ell}_{\alpha}$ and in  $F\pla$, for any $1\le p<\infty$.
\end{enumerate}
\end{lem}

\begin{proof}
The first three assertions are a consequence of the definition of
$E$ and the fact that $e^{\beta|w|^\ell-\gamma |w|^{2\ell}}\in
L^1$, for any  $\beta,\gamma>0$.

The fourth assertion is a consequence of Proposition \ref{prop:pointwise}.

The density of $E$ in $F\pla$ is a consequence of the fact that the holomorphic polynomials are dense in $F\pla$
(see \cite[Theorem~28]{ConsPelJGA2016}).
\end{proof}

In order to define the small Hankel operator for a large
class of symbols we consider the space
$X^{\infty,\ell}_\alpha$ of all measurable functions
$\varphi$ on $\C$ such that
\[
\|\varphi\|_{X^{\infty,\ell}_\alpha}:=\esssup_{z\in\C}|\varphi(z)|(1+|z|)^{2-2\ell}e^{-\frac{\alpha}2|z|^{2\ell}}
<\infty.
\]
Observe that $H^{\infty,\ell}_\alpha=H(\C)\cap X^{\infty,\ell}_\alpha.$

Let $\varphi$ be a function in $X^{\infty,\ell}_\alpha$.
Since $X^{\infty,\ell}_\alpha\subset L^{1,\ell}_{\beta}$,
for any $\beta>\alpha$,
the {\em small Hankel operator} $\hvla$ with symbol $\varphi$ is well defined
on $E$ by
\begin{equation}\label{eqn:def:small:Hankel}
\mathfrak{h}^\ell_{\varphi,\alpha}(f)(z):=\overline{P\la(\overline{f}
\varphi)(z)}=
\int_{\C}K\la(w,z)f(w)\,\overline{\varphi(w)}\,e^{-\alpha|w|^{2\ell}}\,d\nu(w).
\end{equation}

The next proposition states the relationship between the the small Hankel operator $\hvla$ and the
corresponding Hankel bilinear form defined by
\[
\Lambda^\ell_{\varphi,\alpha}(f,g):=
\langle fg,\varphi\rangle\la\qquad(f,g\in E).
\]

\begin{prop}\label{prop:hankelform}
If  $f, g\in E$ and $\varphi\in X^{\infty,\ell}_\alpha$, then we have
\begin{equation}\label{eqn:adjointh}
\Lambda^\ell_{\varphi,\alpha}(f,g)=\langle g,\overline{\hvla(f)}\rangle\la =\langle f,\overline{\hvla(g)}\rangle\la.
\end{equation}
Moreover, if $b=P\la(\varphi)\in H^{\infty,\ell}_\alpha$, then
$\hvla(f)=\mathfrak{h}^\ell_{b,\alpha}(f)$,
$\Lambda^\ell_{b,\alpha}(f,g)=
\Lambda^\ell_{\varphi,\alpha}(f,g)$ and
$\hvla(f)=\mathfrak{h}^\ell_{b,\alpha}(f)$.
\end{prop}

\begin{proof}
    Formula~{\eqref{eqn:adjointh}} follows from Fubini's theorem and the fact
that
\begin{equation*}
\Psi_{f,g,\varphi}(z,w):=K\la(w,z)f(w)\,\overline{\varphi}(w)\,
e^{-\alpha|w|^{2\ell}}g(z) \,e^{-\alpha|z|^{2\ell}}
\end{equation*}
is in $L^1(\C\times\C)$.

This is a consequence of
Proposition~{\ref{prop:pointwise}}. Indeed, if $\lambda>0$ we have that
\begin{align*}
|\Psi_{f,g,\varphi}(w,z)|
&\lesssim \|\varphi\|_{X^{\infty,\ell}_\alpha}
(1+|w|)^{3\ell-3}|f(w)|(1+|z|)^{\ell-1}|g(z)|
 e^{\alpha|z|^\ell|w|^\ell-\frac{\alpha}2|w|^{2\ell}-\alpha|z|^{2\ell}}\\
&\lesssim \|\varphi\|_{X^{\infty,\ell}_\alpha}
e^{\beta |w|^\ell} e^{\beta |z|^\ell}
e^{\frac{\alpha}2(\frac1{\lambda^2}-1)|w|^{2\ell}
    +\alpha(\frac{\lambda^2}2-1)|z|^{2\ell}}
\end{align*}
for some $\beta>0$.
Therefore by choosing $1<\lambda<\sqrt{2}$ we see that
 $\Psi_{f,g,\varphi}\in L^1(\C\times\C)$.

By Lemma~{\ref{lem:spaceE}}, if $f,\, g\in E$ then
$fg\in E\subset F^{1,\ell}_\alpha$, and so $fg=P\la(fg)$, by Proposition~{\ref{prop:Interp-dual}}. Therefore
\begin{align*}
\Lambda^\ell_{\varphi,\alpha}(f,g)
&=\int_\C P\la(fg)(w)\overline{\varphi(w)}e^{-\alpha|w|^{2\ell}} d\nu(w)\\
&=\int_\C \int_\C(fg)(z)K\la(w,z)e^{-\alpha|z|^{2\ell}}d\nu(z)
\overline{\varphi(w)}e^{-\alpha|w|^{2\ell}} d\nu(w).
\end{align*}
Since
$$
(fg)(z)K\la(w,z)e^{-\alpha|z|^{2\ell}}
\overline{\varphi(w)}e^{-\alpha|w|^{2\ell}}
=\Psi_{1,fg,\varphi}(w,z)\in L^1(\C\times\C),
$$
Fubini's theorem gives
$\Lambda^\ell_{b,\alpha}(f,g)=\Lambda^\ell_{\varphi,\alpha}(f,g)$
and $\hvla(f)=\mathfrak{h}^\ell_{b,\alpha}(f)$ for any $f,\,g\in E$.
\end{proof}

As a consequence of the above proposition and
Proposition
 \ref{prop:Interp-dual}\eqref{item:Interp-dual3}-\eqref{item:Interp-dual4} we obtain:

\begin{cor}\label{cor:hankelform} \quad\par
\begin{enumerate}
\item
If $1<p<\infty$, the Hankel operator $\hvla$ defined on the space $E$
extends to a bounded operator on $F\pla$, also denoted by $\hvla$,  if and only if  the
bilinear form $\Lambda^\ell_{\varphi,\alpha}$
defined on $E\times E$ extends to a bounded bilinear
form on $F\pla\times F^{p',\ell}_\alpha$.
Moreover, $\|\hvla\|_{F\pla}\simeq
\|\Lambda^\ell_{\varphi,\alpha}\|_{F\pla\times F^{q,\ell}_\alpha}$.

\item
The Hankel operator $\hvla$ defined on the space $E$
extends to a bounded operator, also denoted by $\hvla$, either on $F^{1,\ell}_\alpha$ or on $\mathfrak{f}^{\infty,\ell}_{\alpha}$ if and only if  the
bilinear form $\Lambda^\ell_{\varphi,\alpha}$
defined on $E\times E$ extends to a bounded bilinear
form on  $F^{1,\ell}_\alpha\times\mathfrak{f}^{\infty,\ell}_{\alpha}$.
Moreover, $\|\hvla\|_{F^{1,\ell}_\alpha}\simeq
 \|\Lambda^\ell_{\varphi,\alpha}\|_{F^{1,\ell}_\alpha\times\mathfrak{f}^{\infty,\ell}_{\alpha}}$.

 \item
 The adjoint (in the sense of \eqref{eqn:adjointh}) of $\hvla:F\pla\to\overline{F\pla}$, $1<p<\infty$, is $\hvla:F^{p',\ell}_\alpha\to\overline{F^{p',\ell}_\alpha}$ and
the adjoint of $\hvla:\mathfrak{f}^{\infty,\ell}_{\alpha}\to\overline{\mathfrak{f}^{\infty,\ell}_{\alpha}}$  is $\hvla:F^{1,\ell}_\alpha\to\overline{F^{1,\ell}_\alpha}$.
 \end{enumerate}
\end{cor}

The last result of this subsection shows that the dilation
operators $\Phi^{\ell}_{\lambda}$, defined
by~{\eqref{eqn:dilations:act:isometrically}}, ``conmute''
with the small Hankel operators.

\begin{prop}\label{prop:dilations:small;Hankel:operator}
Let $1\le p<\infty$, $\alpha,\lambda>0$ and $\ell\in\N$.
Then:
\begin{enumerate}
    \item \label{item:dilations:small;Hankel:operator:1}
    $\Phi^{\ell}_{\lambda}(X^{\infty,\ell}_\alpha)=
    X^{\infty,\ell}_{\lambda\alpha}$ and
    $\Phi^{\ell}_{\lambda}(E)=E$.
    \item \label{item:dilations:small;Hankel:operator:2}
If $\varphi\in X^{\infty,\ell}_\alpha$ and
$\psi=\Phi^{\ell}_{\lambda}\varphi$ then
$\Phi^{\ell}_{\lambda}(\hvla f)=
\mathfrak{h}^\ell_{\psi,\lambda\alpha}
(\Phi^{\ell}_{\lambda}f)$, for every $f\in E$, and  so
$\|\hvla\|_{F\pla}=
\|\mathfrak{h}^\ell_{\psi,\lambda\alpha}
\|_{F^{p,\ell}_{\lambda\alpha}}$ and $\|\hvla\|_{\mathfrak{f}^{\infty,\ell}_{\alpha}}=
\|\mathfrak{h}^\ell_{\psi,\lambda\alpha}
\|_{\mathfrak{f}^{\infty,\ell}_{\alpha}}$.
\end{enumerate}
\end{prop}

\begin{proof}
The proof
of~{\eqref{item:dilations:small;Hankel:operator:1}}
is straightforward.
Part~{\eqref{item:dilations:small;Hankel:operator:2}}
follows from
Proposition~{\ref{prop:dilations:act:isometrically}} and
Proposition~{\ref{prop:dilations:conmute:Bergman:projection}}.
Indeed,
\begin{align*}
\Phi^{\ell}_{\lambda}(\hvla f)
&=\Phi^{\ell}_{\lambda}
(\overline{P^{\ell}_{\alpha}(f\overline{\varphi})})=
\overline{\Phi^{\ell}_{\lambda}
 (P^{\ell}_{\alpha}(f\overline{\varphi}))}\\
&\stackrel{(*)}{=}
\overline{P^{\ell}_{\lambda\alpha}(\Phi^{\ell}_{\lambda}
    (f\overline{\varphi}))}=
\overline{P^{\ell}_{\lambda\alpha}
    ((\Phi^{\ell}_{\lambda}f)\,\overline{\psi})}=
\mathfrak{h}^{\ell}_{\psi,\lambda\alpha}
(\Phi^{\ell}_{\lambda}f),
\end{align*}
for every $f\in E$, where $(*)$ holds by
Proposition~{\ref{prop:dilations:conmute:Bergman:projection}}. Then the above identity and
Proposition~{\ref{prop:dilations:act:isometrically}} directly imply that
$\|\hvla\|_{F\pla}=
\|\mathfrak{h}^\ell_{\psi,\lambda\alpha}
\|_{F^{p,\ell}_{\lambda\alpha}}$.
\end{proof}

\section{Proof of Theorem~{\ref{thm:main1}}}
\label{sect:ProofThm1}

\subsection{Proof of the sufficiency}

\begin{lem}\label{lem:sufcond}
If $\varphi\in L^\infty$ and $1\le p\le \infty$, then $\hvla$ is bounded on $F\pla$  and
 $\|\hvla\|\lesssim \|\varphi\|_{L^\infty}$.
\end{lem}

\begin{proof}
If $\varphi\in L^\infty$ and $f\in F\pla$, then $\varphi f\in L\pla$ and $\|\varphi f\|_{L\pla}\le\|\varphi\|_{\infty}\|f\|_{F\pla}$.
 By Proposition~{\ref{prop:Interp-dual}\eqref{item:Interp-dual1}} $P\la$ is bounded on $L\pla$, so we conclude that
\[
\|\hvla(f)\|_{L\pla}=\|P\la(\varphi\overline{f})\|_{F\pla}
\lesssim\|P\la\|_{L\pla}\,\|\varphi\|_{L^\infty}\|f\|_{F\pla}.\qedhere
\]
\end{proof}

The following result is a corollary of
Proposition~{\ref{prop:Interp-dual}\eqref{item:Interp-dual1}}.

\begin{prop}\label{prop:PtaLinf}
The projection $P\la$ is  bounded  from $L^\infty$ onto $F^{\infty,\ell}_{\alpha/2}$.
Moreover, $\inf\{\|\varphi\|_{L^\infty}: \varphi\in L^{\infty}, P\la \varphi=f\}\le
2^{1/\ell}\|f\|_{F^{\infty,\ell}_{\alpha/2}}$,
 for every $f\in F^{\infty,\ell}_{\alpha/2}$.
\end{prop}

\begin{proof}
It is clear that
$T_{\alpha}\varphi(z):=\varphi(z)e^{\alpha|z|^{2\ell}}$
defines a linear isometry from $L^{\infty}$ onto
$L^{\infty,\ell}_{2\alpha}$. Then,
for any $\varphi\in L^\infty$, we have that
\begin{align*}
P\la\varphi(z)
&=\int_\C K\la(z,w)\, T_{\alpha}\varphi(w)\, e^{-2\alpha|w|^{2\ell}} d\nu(w)\\
&\stackrel{(1)}{=}2^{-1/\ell}\int_\C K^{\ell}_{2\alpha}(2^{-1/\ell}z,w) \,
T_{\alpha}\varphi(w)\, e^{-2\alpha|w|^{2\ell}} d\nu(w) \\
&=2^{-1/\ell}\,P^{\ell}_{2\alpha}(T_{\alpha}\varphi)(2^{-1/\ell}z)
\stackrel{(2)}{=}
2^{-1/\ell}\,\Phi^{\ell}_{1/4}(P^{\ell}_{2\alpha}(T_{\alpha}\varphi))(z)
\end{align*}
where $(1)$ and $(2)$ follow from~{\eqref{eqn:kernel}}
and~{\eqref{eqn:dilations:act:isometrically}}, respectively.
In other words, the projection $P\la$ on $L^{\infty}$ is the
composition of the following three bounded linear exhaustive
 operators:
\begin{enumerate}
\item $T_{\alpha}:L^{\infty}\to L^{\infty,\ell}_{2\alpha}$;
\item $P^{\ell}_{2\alpha}:L^{\infty,\ell}_{2\alpha}\to F^{\infty,\ell}_{2\alpha}$;
\item $\Psi:=2^{-1/\ell}\,\Phi^{\ell}_{1/4}:F^{\infty,\ell}_{2\alpha}\to
F^{\infty,\ell}_{\alpha/2}$.
\end{enumerate}
It directly follows that $P^{\ell}_{\alpha}$ is bounded from
$L^{\infty}$ onto $F^{\infty,\ell}_{\alpha/2}$.
Moreover, since $P^{\ell}_{2\alpha}$ is a projection from
$L^{\infty,\ell}_{2\alpha}$ onto  $F^{\infty,\ell}_{2\alpha}$
(by Proposition~{\ref{prop:Interp-dual}\eqref{item:Interp-dual1}}) and the operator
$\Psi:=2^{-1/\ell}\,\Phi^{\ell}_{1/4}:F^{\infty,\ell}_{2\alpha}\to
F^{\infty,\ell}_{\alpha/2}$ is an
isomorphism such that $\Psi^{-1}=2^{1/\ell}\Phi^{\ell}_4$
satisfies
$\|\Psi^{-1}(f)\|_{F^{\infty,\ell}_{2\alpha}}=
2^{1/\ell}\|f\|_{F^{\infty,\ell}_{\alpha/2}}$, for every
$f\in F^{\infty,\ell}_{\alpha/2}$,
(by Proposition~{\ref{prop:dilations:act:isometrically}})
we conclude that
\[
\inf\{\|\varphi\|_{L^\infty}: \varphi\in L^{\infty}, P\la \varphi=f\}\le
2^{1/\ell}\|f\|_{F^{\infty,\ell}_{\alpha/2}},
\quad\mbox{for every $f\in F^{\infty,\ell}_{\alpha/2}$.}\qedhere
\]
\end{proof}

\begin{prop}\label{prop:sufcond}
Let $b\in F^{\infty,\ell}_{\alpha/2}$.
\begin{enumerate}
    \item \label{item:sufcond1} If $1\le p< \infty$, then
$\hbla$ extends to a bounded operator on $F\pla$.
\item \label{item:sufcond2}
$\hbla$ extends to a bounded operator on $\mathfrak{f}^{\infty,\ell}_\alpha$.
\end{enumerate}
Moreover,
$\|\hbla\|_{F^{p,\ell}_{\alpha}}\lesssim\|b\|_{F^{\infty,\ell}_{\alpha/2}}$, for any
$1\le p<\infty$, and $\|\hbla\|_{\mathfrak{f}^{\infty,\ell}_{\alpha}}\lesssim
\|b\|_{F^{\infty,\ell}_{\alpha/2}}$.
\end{prop}

\begin{proof}
In order to prove \eqref{item:sufcond1}, we show that  for $1\le p\le\infty$,
\begin{equation}\label{eq:suf2}
\|\hbla(f)\|_{L\pla}\lesssim
\|b\|_{F^{\infty,\ell}_{\alpha/2}}\|f\|_{F\pla}
\qquad(f\in E).
\end{equation}
By Proposition \ref{prop:PtaLinf}, $b=P\la(\varphi)$ for some
$\varphi\in L^\infty$ such that $\|\varphi\|_{L^\infty}\le3\|b\|_{F^{\infty,\ell}_{\alpha/2}}$.
If $f\in E$, Proposition
\ref{prop:hankelform} gives $\hbla(f)=\hvla(f)$,
 so Lemma~{\ref{lem:sufcond}} implies \eqref{eq:suf2}.

Taking into account \eqref{eq:suf2} for $p=\infty$, the proof of \eqref{item:sufcond2} will follow after checking $\hbla(E)\subset \mathfrak{f}^{\infty,\ell}_\alpha$. Indeed,
by Proposition \ref{prop:pointwise}, for $f\in E$ and $0<\lambda<1$, we have
\begin{align*}
    |\hbla(f)(z)|
    &\lesssim \int_\C (1+|z|)^{\ell-1}(1+|w|)^{\ell-1}e^{\alpha|z|^\ell|w|^\ell}
    e^{\beta|w|^\ell}   e^{\alpha|w|^{2\ell}/4}
    e^{-\alpha|w|^{2\ell}}d\nu(w)\\
    &   \lesssim (1+|z|)^{\ell-1}\int_\C e^{\alpha\lambda^2|z|^{2\ell}/2}
    e^{\alpha|w|^{2\ell}/(2\lambda^2)}
        e^{2\beta|w|^\ell}
        e^{-3\alpha|w|^{2\ell}/4}d\nu(w)\\
        &\lesssim (1+|z|)^{\ell-1}
     e^{\alpha\lambda^2 |z|^{2\ell}/2}  \int_\C
        e^{2\beta |w|^\ell} e^{-\alpha(3/2-1/\lambda^2)|w|^{2\ell}/2}
    d\nu(w).
\end{align*}
Choosing $\sqrt{2/3}<\lambda<1$, the last integral is finite and we get
\begin{equation*}
\lim_{|z|\to \infty}    |\hbla(f)(z)|e^{-\alpha|z|^{2\ell}/2}=0.\qedhere
\end{equation*}
\end{proof}

\subsection{Proof of the necessity}\quad\par
\label{subsec:proof:necessity:boundedness}

In order to prove the necessity we need some technical results.

The first one is a simple consequence of  Stirling's formula.

\begin{lem}\label{lem:Stirling}
	Let $\delta$  be a positive number. Then
	
	\begin{enumerate}
		\item\label{item:Stirling1}
		$\Gamma(s+t)\simeq s^t\,\Gamma(s)\qquad(s\ge 2\delta,\,|t|\le \delta).$
		\vspace*{3pt}
		\item\label{item:Stirling2} Let $a$ be a real number. Then
		\[
		\sum_{k=0}^{\infty}\frac{s^k}{k!}\frac1{(k+1)^a}\simeq \frac{e^s}{(1+s)^a}
		\qquad(s\ge0).
		\]
	\end{enumerate}

	All the constants in the above equivalences  only depend on $\delta$ and $a$.
\end{lem}

\begin{proof}\eqref{item:Stirling1}
	Stirling's formula gives
	\[
	\Gamma(x)\simeq x^{x-1/2}e^{-x}\qquad(x\ge \delta),
	\]
	so
	\[
	\Gamma(s+t)\simeq (s+t)^{s+t-1/2}e^{-s-t}
	\simeq (s+t)^t (s+t)^{s-1/2} e^{-s}.
	\]
	Since $\frac{s}{2}\le s+t\le 2s$ and $|t|\le\eta$, we have
	$(s+t)^t\simeq s^t$ and
	$(s+t)^{s-1/2}\simeq s^{s-1/2}$.
	
	\noindent
	\eqref{item:Stirling2} Note that both terms
	of the estimate are positive continuous functions of
	$s\ge0$. So it is clear that we only have to prove that
	\begin{equation}\label{eqn:exponential:estimate}
	f_a(x):=\sum_{k=0}^{\infty}\frac{s^k}{k!}\frac{s^a}{(k+1)^a}
	\simeq e^s \qquad(s\ge1).
	\end{equation}
	But
	\[
	f_{a+1}(s)
	=\sum_{k=0}^{\infty}\frac{s^{k+1}}{(k+1)!}\frac{s^a}{(k+1)^a}
	=\sum_{k=1}^{\infty}\frac{s^k}{k!}\frac{s^a}{k^a}
	\simeq \sum_{k=1}^{\infty}\frac{s^k}{k!}\frac{s^a}{(k+1)^a}
	\simeq f_a(s),
	\]
	and so we may assume that $0\le a<1$.
	Let $s\ge1$ and let $j\in\N$ be its integer part.   Then
	\[
	f_a(s)=\sum_{k=0}^{j-1}\frac{s^k}{k!}\left(\frac{s}{k+1}\right)^a
	+\sum_{k=j}^{\infty}\frac{s^k}{k!}\left(\frac{s}{k+1}\right)^a.
	\]
	Now
	\[
	1\le\left(\frac{s}{k+1}\right)^a\le\frac{s}{k+1}
	\qquad(0\le k<j)
	\]
	and
	\[
	\frac{s}{k+1}\le\left(\frac{s}{k+1}\right)^a\le1
	\qquad(j\le k).
	\]
	It follows that
	\[
	\sum_{k=0}^{j-1}\frac{s^k}{k!}
	+\sum_{k=j}^{\infty}\frac{s^{k+1}}{(k+1)!}\le
	f_a(s)\le
	\sum_{k=0}^{j-1}\frac{s^{k+1}}{(k+1)!}
	+\sum_{k=j}^{\infty}\frac{s^k}{k!},
	\]
	and therefore
	\[
	e^s\left(1-\frac2{e}\right)\le e^s\left(1-\frac{(j+1)^j}{e^j\,j!}\right)\le e^s-\frac{s^j}{j!}\le
	f_a(s)\le 2e^s,
	\]
	since the sequence $c_j=\frac{(j+1)^j}{e^j\,j!}$ is
	decreasing. Hence \eqref{eqn:exponential:estimate} holds.
\end{proof}

The  following lemma is an essential tool to prove the necessity.

\begin{lem} \label{lem:estimate}
	For $\ell\in\N$, $a,b> 0$ and $c\ge 0$, let
	\[
	{\mathcal I}^{\ell}_{a,b,c}(z):=
	\int_\C\bigl|e^{a(z\overline{w})^\ell}\bigr|^2 e^{-b|w|^{2\ell}} (1+|w|)^{c} d\nu(w)\qquad(z\in\C).
	\]
	Then
	\begin{equation}\label{eqn:estimateI}
	{\mathcal I}^{\ell}_{a,b,c}(z)
	\simeq e^{a^2|z|^{2\ell}/b} (1+|z|)^{c+2-2\ell}
	\qquad(z\in\C).
	\end{equation}
\end{lem}

\begin{proof}
	It is enough to prove the estimate \eqref{eqn:estimateI} for $|z|\ge 1$.
	Observe that
	\[
	{\mathcal I}^{\ell}_{a,b,c}(z)
	\simeq 		{\mathcal J}^{\ell}_{a,b,0}(z)+{\mathcal J}^{\ell}_{a,b,c}(z),
	\]
	where
	\[
	{\mathcal J}^{\ell}_{a,b,c}(z):=
	\int_\C\bigl|e^{a(z\overline{w})^\ell}\bigr|^2 e^{-b|w|^{2\ell}} |w|^{c} d\nu(w).
	\]
	
	Thus we only have to show that
	\[
	{\mathcal J}^{\ell}_{a,b,c}(z)
	\simeq e^{a^2|z|^{2\ell}/b} |z|^{c+2-2\ell}
	\qquad(|z|\ge 1).
	\]
	Indeed, by integrating in polar coordinates and orthogonality,
	\begin{align*}
	{\mathcal J}^{\ell}_{a,b,c}(z)
	&\simeq\sum_{k=0}^\infty \int_0^\infty
	\frac{a^{2k}|z|^{2k\ell}}{(k!)^2}e^{-br^{2\ell}}r^{2k\ell+c+1}dr\\
	&\simeq\sum_{k=0}^\infty
	\frac{a^{2k}|z|^{2k\ell}}{b^{k+(c+2)/(2\ell)}}
	\frac1{(k!)^2}\int_0^\infty e^{-t}t^{(2k\ell+c+2)/(2\ell)-1}dt\\
	&\simeq\sum_{k=0}^\infty
	\frac{a^{2k}|z|^{2k\ell}}{b^{k}}
	\frac{\Gamma(k+(c+2)/(2\ell))}{(k!)^2}.
	\end{align*}
	Therefore Lemma \ref{lem:Stirling} completes the proof:
	\[
	{\mathcal J}^{\ell}_{a,b,c}(z)\simeq
	\sum_{k=0}^\infty
	\frac{a^{2k}|z|^{2k\ell}}{b^{k}}
	\frac{1}{k!(k+1)^{(2\ell-2-c)/(2\ell)}}
	\simeq e^{a^2|z|^{2\ell}/b} |z|^{c+2-2\ell}.\qedhere
	\]
\end{proof}

\begin{proof}[Proof of the necessity]
	Let $1\le p\le \infty$ and $b\in H^{\infty,\ell}_{\alpha}$.
	Suppose that $\hbla:(E,\|\cdot\|_{F\pla})\to L\pla$ is bounded and we want to prove that
	$b\in F^{\infty,\ell}_{\alpha/2}$ and $\|b\|_{F^{\infty,\ell}_{\alpha/2}}\lesssim\|\hbla\|_{F\pla}$.
	
	First of all,
	by Proposition~{\ref{prop:dilations:small;Hankel:operator}}
	we may assume that $\alpha=1$.
	Now \eqref{eqn:reproducing:property} gives that
	\begin{equation} \label{eqn:reproducing:property:bis}
	\overline{b(z)}=\int_\C
	K^\ell_1(w,z)\,\overline{b(w)}\,e^{-|w|^{2\ell}}\,d\nu(w)
	=\langle K^{\ell}_1(\cdot,z),\,b\rangle^{\ell}_1.
	\end{equation}
	
	We decompose the Bergman kernel as
	\begin{equation*}\label{eqn:bergmankernel:decomposition}
	K^\ell_1(w,z)   = G_0(w,z)G_1(w,z),
	\end{equation*}
	where
	\begin{equation}\label{eqn:defnG0G1}
	G_0(w,z):= e^{\frac{(w\overline z)^\ell}{2}}\quad\text{and}\quad G_1(w,z):= e^{-\frac{(w\overline z)^\ell}{2}}K^\ell_1(w,z).
	\end{equation}

	By Proposition \ref{prop:pointwise},  $G_{0}(\cdot,z), G_1(\cdot,z)\in E$, and so \eqref{eqn:reproducing:property:bis} and
	Proposition~{\ref{prop:hankelform}} show
	\begin{equation}\label{eqn:symbol:representation}
	\overline{b(z)}=
	\langle G_1(\cdot,z),\,
	\overline{\mathfrak{h}^{\ell}_{b,1}(G_0(\cdot,z))}
	\rangle^{\ell}_1.
	\end{equation}
	
	Therefore the boundedness of $\mathfrak{h}^{\ell}_{b,1}$ implies that
	\begin{equation}\label{eqn:estimate:boundedness}
	|b(z)|
	\lesssim
	\|\mathfrak{h}^{\ell}_{b,1}\|_{F^{p,\ell}_1}
	\| G_0(\cdot,z)\|_{F^{p,\ell}_{1} }
	\|G_{1}(\cdot,z)\|_{F^{p',\ell}_{1}}.
	\end{equation}

	We claim that:
	\begin{eqnarray}
	\| G_{0}(\cdot,z)\|_{F^{p,\ell}_{1} }&\simeq& (1+|z|)^{2(1-\ell)/p}\,e^{|z|^{2\ell}/8}
	\label{eqn:G0:estimate}\\
	\| G_{1}(\cdot,z)\|_{F^{p',\ell}_{1} }&\lesssim& (1+|z|)^{2(\ell-1)/p}\,e^{|z|^{2\ell}/8}
	\label{eqn:G1:estimate}
	\end{eqnarray}
	
	These norm-estimates together
	with~{\eqref{eqn:estimate:boundedness}} give
	$|b(z)|\lesssim \|\hbla\|_{F^{p,\ell}_1}\, e^{|z|^{2\ell}/4}$.
	
	Now,  for $1\le p<\infty$, \eqref{eqn:G0:estimate}	is a consequence of Lemma \ref{lem:estimate}:
	\begin{align*}
	\| G_{0}(\cdot,z)\|^p_{F^{p,\ell}_{1} }&=
	\int_\C\left|e^{p(z\overline w)^\ell/4}\right|^2e^{-p|w|^{2\ell}/2} d\nu(w)
	={\mathcal I}^{\ell}_{p/4,p/2,0}(z)\\
	&\simeq (1+|z|)^{2(1-\ell)}\,e^{p|z|^{2\ell}/8}.
	\end{align*}
	
	If $p=\infty$, using the identity
	\begin{equation}\label{eqn:completarquadrats}
	\Re((z\overline{w})^\ell)-|w|^{2\ell}=-|w^\ell-z^\ell/2|^2+|z|^{2\ell}/4,
	\end{equation}
	we obtain
	\[
	\| G_{0}(\cdot,z)\|_{F^{\infty,\ell}_{1} }=\sup_{w\in\C}|e^{(z\overline w)^\ell/2}|e^{-|w|^{2\ell}/2}= \,e^{|z|^{2\ell}/8}.
	\]
	
	On the other hand, by Proposition   \ref{prop:pointwise}
\begin{align*}
|G_1(w,z)| &\lesssim
(1+|zw|)^{\ell-1}\left(
e^{\Re((z\overline w)^\ell)/2}+
e^{-\Re((z\overline w)^\ell)/2}
\right)
\\ & \lesssim
(1+|z|)^{\ell-1}(1+|w|)^{\ell-1}
\left(
e^{\Re((z\overline w)^\ell)/2}+
e^{-\Re((z\overline w)^\ell)/2}
\right).
\end{align*}
	
	Therefore, for $1\le p'<\infty$, we have
	\[
	\| G_{1}(\cdot,z)\|_{F^{p',\ell}_{1} }\lesssim J_1(z)^{1/p'}+J_2(z)^{1/p'},
	\]
	where
	\[
	J_1(z):=	(1+|z|)^{p'(\ell-1)}\int_\C\left|e^{p'(z\overline w)^\ell/4}\right|^2e^{-p'|w|^{2\ell}/2}
	(1+|w|)^{p'(\ell-1)}d\nu(w)
	\]
	and
	\[
	J_2(z):=	(1+|z|)^{p'(\ell-1)}\int_\C\left|e^{-p'(z\overline w)^\ell/4}\right|^2e^{-p'|w|^{2\ell}/2}
	(1+|w|)^{p'(\ell-1)}d\nu(w).
	\]
	By Lemma \ref{lem:estimate},
	\[
	J_1(z)
	=(1+|z|)^{p'(\ell-1)}\,{\mathcal I}^{\ell}_{p'/4,p'/2,p'(\ell-1)}(z)\\
	\simeq (1+|z|)^{2(p'-1)(\ell-1)}\,e^{p'|z|^{2\ell}/8}.
	\]
	Since $J_2(z)=J_1(e^{i\pi/\ell}z)$,
	we obtain the estimate \eqref{eqn:G1:estimate}.
	
	If $p'=\infty$, by using \eqref{eqn:completarquadrats},
	we have
	\[
	\| G_{1}(\cdot,z)\|_{F^{\infty,\ell}_{1} }=\sup_{w\in\C}|G_1(w,z)|e^{-|w|^{2\ell}/2}\lesssim (1+|z|)^{2(\ell-1)}\,e^{|z|^{2\ell}/8}.\qedhere
	\]
\end{proof}

\section{Proof of Theorem \ref{thm:main3}}
\label{sect:ProofThm3}

The next proposition will be used to prove Theorem \ref{thm:main3}.

\begin{prop}\label{prop:duall12a}
	The dual of $F^{1,\ell}_{2\alpha}$ with respect to the pairing $\langle\cdot,\cdot\rangle\la$
	is $F^{\infty,\ell}_{\alpha/2}$.
\end{prop}

\begin{proof}
	By Proposition \ref{prop:Interp-dual}, if $\Phi\in \left(F^{1,\ell}_{2\alpha}\right)^*$, there exists a unique
	$h\in F^{\infty,\ell}_{2\alpha}$ such that
	$$
	\Phi(f)=\langle f,h\rangle^\ell_{2\alpha}=\langle f(z),h(z)e^{-\alpha|z|^{2\ell}}\rangle\la,\quad\mbox{for any $f\in E$.}
	$$
	Since  $\varphi(z)=h(z)e^{-\alpha|z|^{2\ell}}\in L^\infty$, Proposition \ref{prop:PtaLinf}  gives
	$g=P\la(\varphi)\in F^{\infty,\ell}_{\alpha/2}$, so
	$$
	\Phi(f)=\langle f,g \rangle\la,\quad\mbox{for any $f\in E$.}
	$$
	
	Conversely, if $g\in F^{\infty,\ell}_{\alpha/2}$,
	by Proposition \ref{prop:PtaLinf} there exists $\varphi\in L^\infty$ such that
	$P\la(\varphi)=g$ and $\|\varphi\|_{L^\infty}\simeq \|g\|_{F^{\infty,\ell}_{\alpha/2}}$.
	Thus, for $f\in E$, we have
	$$
	|\langle f,g\rangle\la|=|\langle f,\varphi\rangle\la|\le \|\varphi\|_{L^\infty}\|f\|_{F^{1,\ell}_{2\alpha}}.
	$$
	
	This ends the proof.
\end{proof}

\begin{proof}[Proof of Theorem \ref{thm:main3}]
	The proof of this result follows from standard arguments used in the setting  of
	classical spaces of holomorphic functions.   We only include  a sketch of the proof
	for the sake of completeness.

	\begin{enumerate}
		\item  It is a consequence of Corollary \ref{cor:hankelform} and Theorem
		\ref{thm:main1}.
		
		\item
		It is a consequence of Propositions
		\ref{prop:PtaLinf} and \ref{prop:duall12a}.
		
		\item
		First we consider the case $1<p<\infty$. By~{\eqref{item:main32}}, in order to show that  $F^{p,\ell}_\alpha\odot F^{p',\ell}_\alpha=F^{1,\ell}_{2\alpha}$, it is enough to prove that the dual of
		$F^{p,\ell}_\alpha\odot F^{p',\ell}_\alpha$ with respect to the pairing $\langle \cdot,\,\cdot\rangle^{\ell}_{\alpha}$ is $F^{\infty,\ell}_{\alpha/2}$.
		
		By~{\eqref{item:main31}},
		if $b\in F^{\infty,\ell}_{\alpha/2}$ then $\Lambda^{\ell}_{b,\alpha}$ defines a bounded bilinear form on
		$F^{p,\ell}_\alpha\times F^{p',\ell}_\alpha$, so
		$h\mapsto \Lambda^{\ell}_{b,\alpha}(h,1)$ is a bounded linear form on $F^{p,\ell}_\alpha\odot F^{p',\ell}_\alpha$.

		Conversely, it is clear that any form $\Phi$ on $F^{p,\ell}_\alpha\odot F^{p',\ell}_\alpha$
		defines a bounded linear form on
		$F^{p,\ell}_\alpha$. Thus, by
		Proposition~{\ref{prop:Interp-dual}\eqref{item:Interp-dual3},}
		there exists $b\in F^{p',\ell}_\alpha$ such that $\Phi(h)=\Lambda^{\ell}_{b,\alpha}(h,1)$,
		for any $h\in E$.
		Since the space $E$ is dense in
		$F\pla$ and $F^{p',\ell}_{\alpha}$, the bilinear form  $\Lambda^{\ell}_{b,\alpha}$
		extends boundedly to  $F^{p,\ell}_\alpha\odot F^{p',\ell}_\alpha$.
		Thus, by part~{\eqref{item:main31}}, $b\in F^{\infty,\ell}_{\alpha/2}$.
		
		Similar arguments, using
		Proposition~{\ref{prop:Interp-dual}\eqref{item:Interp-dual4},} prove that
		$F^{1,\ell}_\alpha\odot \mathfrak{f}^{\infty,\ell}_\alpha=F^{1,\ell}_{2\alpha}$.
		Since $F^{1,\ell}_\alpha\odot F^{\infty,\ell}_\alpha\subset F^{1,\ell}_{2\alpha}$,
		we have
		$$
		F^{1,\ell}_{2\alpha}=
		F^{1,\ell}_\alpha\odot\mathfrak{f}^{\infty,\ell}_\alpha\subset
		F^{1,\ell}_\alpha\odot F^{\infty,\ell}_\alpha\subset F^{1,\ell}_{2\alpha},
		$$
		which ends the proof.
		\qedhere
	\end{enumerate}
\end{proof}

\section{Proof of Theorem~{\ref{thm:main2}}}
\label{sect:ProofThm2}

In order to prove Theorem~{\ref{thm:main2}} we will use  a standard technique
based on the  following lemma.
\begin{lem}\label{lem:weakly:convergence}
Let $1<p\le\infty$, $\ell\in\N$ and $\alpha>0$. Let
        $\{g_n\}_{n\in\N}$ be a sequence of functions in $E$. Then, the
        following conditions are equivalent:
        \begin{enumerate}
            \item \label{item:weakly:convergence1}
            $g_n\to
            0$ weakly in $F^{p,\ell}_{\alpha}$, if $p<\infty$, and in
            $\mathfrak{f}^{\infty,\ell}_{\alpha}$, if $p=\infty$.
            \item \label{item:weakly:convergence2}
            $g_n\to 0$ uniformly on compact subsets of $\C$ and ${\displaystyle\sup_{n\in\N}\|g_n\|_{F^{p,\ell}_{\alpha}}<\infty}$.
        \end{enumerate}
\end{lem}

\begin{proof}
    Assume that~{\eqref{item:weakly:convergence1}} holds.
    Then it is well known that
    $\sup_{n\in\N}\|g_n\|_{F^{p,\ell}_{\alpha}}<\infty$, so $\{g_n\}$ is uniformly
    bounded on compact subsets of $\C$. Moreover, since $g_n\to 0$
    weakly in $F^{p,\ell}_{\alpha}$, then, for each $z\in\C$,
    \[ g_n(z)=
    \langle g_n, K^{\ell}_{\alpha,z} \rangle^{\ell}_{\alpha}\to 0, \quad
    \mbox{as $n\to \infty$.}\] Consequently, $g_n\to 0$ uniformly on compact subsets
    of $\C$, by Montel's theorem.
    \par Reciprocally, assume that~{\eqref{item:weakly:convergence2}}  holds.
    By
    Proposition~{\ref{prop:Interp-dual}\eqref{item:Interp-dual3}-\eqref{item:Interp-dual4}},
    we have to show that $\langle f,\,g_n\rangle^{\ell}_{\alpha}\to 0$,
    as $n\to\infty$, for every $f\in F^{p',\ell}_{\alpha}$.

    Let $f\in F^{p',\ell}_{\alpha}$. Then, for every $R>0$, we have
    \[
    \langle f,\,g_n\rangle^{\ell}_{\alpha}= \biggl\{\int_{|w|\le
        R}+\int_{|w|>R}\biggr\}
    f(w)g_n(w)e^{-|w|^{2\ell}}d\nu(w)=I_n(R)+J_n(R).
    \]
    Since $p'<\infty$, we have that
    $\int_{|w|>R}|f(w)e^{-|w|^{2\ell}/2}|^{p'}\,d\nu(w)\to0$, as
    $R\to\infty$, so
    \[
    \lim_{R\to\infty}\sup_{n\in\N}J_n(R)=0,
    \]
    by H\"{o}lder's inequality and the fact that
    $\sup_{n\in\N}\|g_n\|_{F^{p,\ell}_{\alpha}}<\infty$. Moreover, since
    $g_n\to 0$ uniformly on compact subsets of $\C$ then $I_n(R)\to0$,
    as $n\to\infty$, for every $R>0$. It turns out that $\langle
    f,\,g_n\rangle^{\ell}_{\alpha}\to 0$, as $n\to\infty$, and the proof
    is complete.
\end{proof}

\begin{proof}[Proof of Theorem~{\ref{thm:main2}}]
 By Proposition~{\ref{prop:dilations:small;Hankel:operator}}
 we only have to prove Theorem~{\ref{thm:main2}} for
 $\alpha=1$.

 First we prove  that, if either $\mathfrak{h}^{\ell}_{b,1}:
 F^{p,\ell}_1\to \overline{F^{p,\ell}_1}$
 , $1<p<\infty$, or
 $\mathfrak{h}^{\ell}_{b,1}:
 \mathfrak{f}^{\infty,\ell}_1\to
 \overline{\mathfrak{f}^{\infty,\ell}_1}$
 is compact, then $b\in \mathfrak{f}^{\infty,\ell}_{1/2}$.

 Suppose that  $\mathfrak{h}^{\ell}_{b,1}:
F^{p,\ell}_1\to L^{p,\ell}_1$ is compact and
 we want to prove that
 $b\in \mathfrak{f}^{\infty,\ell}_{1/2}$.

Let $G_0, G_1$ be the functions defined  by~{\eqref{eqn:defnG0G1}}.

Since $\mathfrak{h}^{\ell}_{b,1}:
F^{p,\ell}_1\to L^{p,\ell}_1$
($\mathfrak{h}^{\ell}_{b,1}:\mathfrak{f}^{\infty,\ell}_
1\to L^{\infty,\ell}_1$)
is bounded,
the proof of the necessity in
Theorem~{\ref{thm:main1}}
(see~{\S\ref{subsec:proof:necessity:boundedness}}) implies
that~{\eqref{eqn:symbol:representation}} holds, and so
\begin{align*}
|b(z)|&\lesssim
\|\mathfrak{ h}^{\ell}_{b,1}(G_{0}(\cdot,z))\|_{L^{p,\ell}_{1}}\,
\|G_{1}(\cdot,z)\|_{F^{p',\ell}_{1}}\\
&=
\|\mathfrak{ h}^{\ell}_{b,1}(g_{0}(\cdot,z))\|_{L^{p,\ell}_{1}}\,
\|G_{0}(\cdot,z)\|_{F^{p,\ell}_{1}}
\|G_{1}(\cdot,z)\|_{F^{p',\ell}_{1}},
\end{align*}
where $g_0(w,z)=G_0(w,z)/\|G_{0}(\cdot,z)\|_{F^{p,\ell}_{1}}$.
Then~{\eqref{eqn:G0:estimate}} and~{\eqref{eqn:G1:estimate}}
show that
\[
|b(z)|e^{-|z|^{2\ell}/4}\lesssim
\|{\mathfrak h}^{\ell}_{b,1}(g_{0}(\cdot,z))\|_{L^{p,\ell}_{1}},
\]

It is easy to check that $g_0(\cdot,z)\to 0$ uniformly
on compact subsets of $\C$, as $|z|\to\infty$. By
Lemma~{\ref{lem:weakly:convergence}} it follows that
$g_0(\cdot,z)\to 0$  weakly in $F^{p,\ell}_1$, as $|z|\to\infty$.
Note that, if $p=\infty$, the same arguments show  that  $g_0(\cdot,z)\to 0$  weakly in $\mathfrak{f}^{\infty,\ell}_1$.

Then the compactness of $\mathfrak{h}^{\ell}_{b,1}:
F^{p,\ell}_1\to L^{p,\ell}_1$, $1<p<\infty$,
($\mathfrak{h}^{\ell}_{b,1}:\mathfrak{f}^{\infty,\ell}_
1\to L^{\infty,\ell}_1$, respectively) shows that
\begin{equation}\label{eqn:norm:convergence:to:0}
\lim_{|z|\to \infty}
\|\mathfrak{h}^{\ell}_{b,1}(g_0(\cdot,z))\|_{L^{p,\ell}_{1}}=0.
\end{equation}

and so~{\eqref{eqn:norm:convergence:to:0}} gives that
$|b(z)|e^{-|z|^{2\ell}/4}\to0$, as $|z|\to\infty$.

Now we consider the case $p=1$.
By Corollary \ref{cor:hankelform}, the operator  $\mathfrak{h}^{\ell}_{b,1}:
F^{1,\ell}_1\to \overline{F^{1,\ell}_1}$
is the adjoint of
$\mathfrak{h}^{\ell}_{b,1}:{\mathfrak f}^{\infty,\ell}_1\to
\overline{{\mathfrak f}^{\infty,\ell}_1}$.
Thus the compactness of the first operator implies  the
compactness of the second operator and, as we have just
shown, this implies that
$b\in \mathfrak{f}^{\infty,\ell}_{1/2}$.

Now assume that $b\in\mathfrak{f}^{\infty,\ell}_{1/2}$ and
$1\le p<\infty$. Then, by Theorem~{\ref{thm:main1}},
$\mathfrak{h}^{\ell}_{b,1}$ is a bounded operator
from $F^{p,\ell}_1$ to $\overline{F^{p,\ell}_1}$.
Moreover, since $\mathfrak{f}^{\infty,\ell}_{1/2}$ is the
closure of the polynomials in $F^{\infty,\ell}_{1/2}$, there
is a sequence of polynomials $\{P_n\}_{n\in\N}$ such that
$\|P_n-b\|_{F^{\infty,\ell}_{1/2}}\to0$.
Therefore
$\|\mathfrak{h}^{\ell}_{b,1}
-\mathfrak{h}^{\ell}_{P_n,1}\|_{F^{p,\ell}_1}\to0$, because
\[
\|\mathfrak{h}^{\ell}_{P_n,1}
-\mathfrak{h}^{\ell}_{b,1}\|_{F^{p,\ell}_1}
=\|\mathfrak{h}^{\ell}_{P_n-b,1}\|_{F^{p,\ell}_1}\lesssim \|P_n-b\|_{F^{\infty,\ell}_{1/2}},
\]
by Theorem~{\ref{thm:main1}} again. Since
$\{\mathfrak{h}^{\ell}_{P_n,1}\}_{n\ge0}$ is a sequence of finite
rank operators, it follows that
$\mathfrak{h}^{\ell}_{b,1}:F^{p,\ell}_{1}\to\overline{F^{p,\ell}_{1}}$
is compact.
\par Note that the above argument also works by replacing the
 space $F^{p,\ell}_{1}$ by $\mathfrak{f}^{\infty,\ell}_1$, and hence the proof of Theorem~{\ref{thm:main2}}
 is complete.
 \end{proof}

\section{Proof of  Theorem~{\ref{thm:main4}}}
\label{sect:ProofThm4}

\subsection{The small Hankel operator on $F\dla$}\quad\par

By Proposition~{\ref{prop:dilations:small;Hankel:operator}}
it is enough to prove the result for $\alpha=1$, that is,
 to prove
\begin{equation}\label{eqn:main:estimate}
\|\mathfrak{h}^{\ell}_{b,1}\|_{\mathcal{S}_2(F^{2,\ell}_1)}^2
\simeq\|b\|^2_{F^{2,\ell}_{1/2,\Delta}}.
\end{equation}

In order to do that, first we estimate
$\|\mathfrak{h}^{\ell}_{b,1}\|_{\mathcal{S}_2(F^{2,\ell}_1)}^2$ and
$\|b\|^2_{F^{2,\ell}_{1/2,\Delta}}$ in terms of the Taylor coefficients of $b$.

\begin{lem}\label{lem:coefficients}
Let $\ell\in \N$  and let $b(z)=\sum_{m=0}^{\infty}c_mz^m$ be a
function in $H^{\infty,\ell}_1$. Then
\begin{equation}\label{eqn:coefficients:eq1}
\|\mathfrak{h}^{\ell}_{b,1}\|_{\mathcal{S}_2(F^{2,\ell}_1)}^2
\simeq\sum_{m=0}^{\infty}|c_m|^2\,\,
\Gamma\Bigl(\frac{m+1}{\ell}\Bigr)^2\,
\sum_{k=0}^m\frac1{\Gamma\bigl(\frac{k+1}{\ell}\bigr)
    \Gamma\bigl(\frac{m-k+1}{\ell}\bigr)},
\end{equation}
and
\begin{equation}\label{eqn:coefficients:eq2}
\|b\|^2_{F^{2,\ell}_{1/2,\Delta}}\simeq
\sum_{m=0}^{\infty}|c_m|^2\,\,2^{m/\ell}\,
\Gamma\Bigl(\frac{m}{\ell}+1\Bigr).
\end{equation}
\end{lem}

\begin{proof}
We begin proving \eqref{eqn:coefficients:eq1}. Let  $e_n(z)=z^n/\|z^n\|_{F^{2,\ell}_1}$,  $n=0,1,\cdots$.
It is easy to check that
\begin{align*}
\overline{\mathfrak{h}^{\ell}_{b,1}(e_n)(z)}
=\sum_{m=0}^\infty c_{n+m}\frac{\|w^{m+n}\|_{F^{2,\ell}_1}^2}{\|w^{m}\|_{F^{2,\ell}_1}\|w^{n}\|_{F^{2,\ell}_1}}\,e_m(z).
\end{align*}
Thus
\begin{align*}
\|\mathfrak{h}^{\ell}_{b,1}\|^2_{S_2(F^{2,\ell}_1)}
&=\sum_{n=0}^\infty\|\mathfrak{h}^{\ell}_{b,1}(e_n)\|^2_{F^{2,\ell}_1}
=\sum_{n=0}^\infty \sum_{m=0}^\infty
|c_{n+m}|^2\frac{\|w^{m+n}\|_{F^{2,\ell}_1}^4}{\|w^{n}\|_{F^{2,\ell}_1}^2\|w^{m}\|_{F^{2,\ell}_1}^2}\\
&=\sum_{m=0}^\infty \sum_{n=0}^m
|c_{m}|^2\frac{\|w^{m}\|_{F^{2,\ell}_1}^4}{\|w^{n}\|_{F^{2,\ell}_1}^2\|w^{m-n}\|_{F^{2,\ell}_1}^2}\\
&=\sum_{m=0}^{\infty}|c_m|^2\,\,
\Gamma\Bigl(\frac{m+1}{\ell}\Bigr)^2\,
\sum_{k=0}^m\frac1{\Gamma\bigl(\frac{k+1}{\ell}\bigr)
    \Gamma\bigl(\frac{m-k+1}{\ell}\bigr)}.
\end{align*}

Next we prove \eqref{eqn:coefficients:eq2}:
\begin{align*}
\|b\|_{F^{2,\ell}_{1/2,\Delta}}^2&=\sum_{m=0}^\infty  |c_m|^2 \int_\C |z|^{2m}e^{-|z|^{2\ell}/2}(1+|z|^{2\ell-2})\,d\nu(z)\\
&\simeq\sum_{m=0}^\infty  |c_m|^2 \int_0^\infty r^{2m+1}(1+r^{2\ell-2})\,e^{-r^{2\ell}/2} dr\\
&\simeq \sum_{m=0}^\infty  |c_m|^2\, 2^{m/\ell}\,
\Bigl\{\Gamma\Bigl(\frac{m+1}{\ell}\Bigr)+
\Gamma\Bigl(\frac{m}{\ell}+1\Bigr)\Bigr\}\\
&\simeq \sum_{m=0}^\infty  |c_m|^2\, 2^{m/\ell}\,\Gamma\bigl(\frac{m}{\ell}+1\bigr).\qedhere
\end{align*}
\end{proof}

From Lemma~{\ref{lem:coefficients}} it is clear that~{\eqref{eqn:main:estimate}} is equivalent to
\[
\Gamma\Bigl(\frac{m+1}{\ell}\Bigr)^2\,
\sum_{k=0}^m\frac1{\Gamma\bigl(\frac{k+1}{\ell}\bigr)
    \Gamma\bigl(\frac{m-k+1}{\ell}\bigr)} \simeq
\,2^{m/\ell}\,
\Gamma\Bigl(\frac{m}{\ell}+1\Bigr)\qquad(m\ge0),
\]
which can be written as
\begin{equation}\label{eqn:estimate:coefficients}
\sum_{k=0}^m
\frac{\Gamma\bigl(\frac{m+2-\ell}{\ell}\bigr)}
       {\Gamma\bigl(\frac{k+1}{\ell}\bigr)
        \Gamma\bigl(\frac{m-k+1}{\ell}\bigr)} \simeq
\,2^{m/\ell}\,
\frac{\Gamma\Bigl(\frac{m+2-\ell}{\ell}\Bigr)
    \Gamma\Bigl(\frac{m}{\ell}+1\Bigr)}
{\Gamma\bigl(\frac{m+1}{\ell}\bigr)^2}\qquad(m\ge8\ell).
\end{equation}
Now, by Stirling's formula,
\[
\frac{\Gamma\Bigl(\frac{m+2-\ell}{\ell}\Bigr)
    \Gamma\Bigl(\frac{m}{\ell}+1\Bigr)}
{\Gamma\bigl(\frac{m+1}{\ell}\bigr)^2}\simeq1
\qquad(m\ge8\ell).
\]
Hence~{\eqref{eqn:estimate:coefficients}} follows from the
following lemma.
\begin{lem}\label{lem:estimate:coefficients}
\[
\sum_{k=0}^m
\frac{\Gamma\bigl(\frac{m+2-\ell}{\ell}\bigr)}
{\Gamma\bigl(\frac{k+1}{\ell}\bigr)
    \Gamma\bigl(\frac{m-k+1}{\ell}\bigr)}\simeq 2^{m/\ell}
\qquad(m\ge8\ell).
\]
\end{lem}

The key ingredient to prove
Lemma~{\ref{lem:estimate:coefficients}} is the following
important inequality.

\begin{Chernoff:inequality}[{\cite[(1.3.10) p.16]{dudley}}]
\[
\sum_{0\le i\le n/4}\binom{n}{i}\le 2^n e^{-{n/8}},
\quad\mbox{ for every $n\ge0$.}
\]
\end{Chernoff:inequality}

\begin{proof}[Proof of
Lemma~{\ref{lem:estimate:coefficients}}]
Let $m=n\ell+r$, where $n\ge8$ and $0\le r<\ell$.
Then we may decompose the sum $S(m)$ of the statement as
\begin{eqnarray*}
S(m)
&=& \sum_{j=0}^{n-1}\sum_{s=0}^{\ell-1}
    \frac{\Gamma\bigl(\frac{m+2-\ell}{\ell}\bigr)}
    {\Gamma\bigl(\frac{j\ell+s+1}{\ell}\bigr)
        \Gamma\bigl(\frac{m-(j\ell+s)+1}{\ell}\bigr)}
   +\sum_{s=0}^r
  \frac{\Gamma\bigl(\frac{m+2-\ell}{\ell}\bigr)}
   {\Gamma\bigl(\frac{n\ell+s+1}{\ell}\bigr)
    \Gamma\bigl(\frac{m-(n\ell+s)+1}{\ell}\bigr)}\\
&=& \sum_{s=0}^{\ell-1}\sum_{j=0}^{n-1}
\frac{\Gamma\bigl(n-1+\frac{r+2}{\ell}\bigr)}
{\Gamma\bigl(j+\frac{s+1}{\ell}\bigr)
    \Gamma\bigl(n-j+\frac{r-s+1}{\ell}\bigr)}
+\sum_{s=0}^r
\frac{\Gamma\bigl(n-1+\frac{r+2}{\ell}\bigr)}
{\Gamma\bigl(n+\frac{s+1}{\ell}\bigr)
    \Gamma\bigl(\frac{r-s+1}{\ell}\bigr)}\\
&=& \sum_{s=0}^{\ell-1}
\frac{\Gamma\bigl(n-1+\frac{r+2}{\ell}\bigr)}
{\Gamma\bigl(\frac{s+1}{\ell}\bigr)
    \Gamma\bigl(n+\frac{r-s+1}{\ell}\bigr)} +
   \sum_{s=0}^r
   \frac{\Gamma\bigl(n-1+\frac{r+2}{\ell}\bigr)}
   {\Gamma\bigl(n+\frac{s+1}{\ell}\bigr)
    \Gamma\bigl(\frac{r-s+1}{\ell}\bigr)}\\
& & + \sum_{s=0}^{\ell-1}
   \Bigl\{\sum_{1\le j\le\frac{n}4}+
          \sum_{\frac{n}4<j<\frac{3n}4}+
          \sum_{\frac{3n}4\le j\le n-1} \Bigr\}
\frac{\Gamma\bigl(n-1+\frac{r+2}{\ell}\bigr)}
{\Gamma\bigl(j+\frac{s+1}{\ell}\bigr)
    \Gamma\bigl(n-j+\frac{r-s+1}{\ell}\bigr)}\\
&=& S_1(m)+S_2(m)+S_3(m)+S_4(m)+S_5(m).
\end{eqnarray*}
In order to estimate the above five sums we recall that
$\Gamma$ is an increasing function on $[2,\infty)$.
Then, since $\frac{r+2}{\ell}\le2$, we have that
\begin{equation}\label{estimate:numerator}
\Gamma\bigl(n-1+\tfrac{r+2}{\ell})\le\Gamma(n+1).
\end{equation}
On the other hand, since
\[
\frac2{\ell}-1\le\frac{r-s+1}{\ell}\le1
\quad\mbox{ and }\quad
\frac1{\ell}\le\frac{s+1}{\ell}\le1
\qquad(0\le s<\ell),
\]
we also have that
\begin{equation}\label{estimate:denominator1}
\Gamma(n-j-1)\le\Gamma\bigl(n-j+\tfrac{r-s+1}{\ell}\bigr)
\qquad(0\le s<\ell,\,0\le j\le n-3),
\end{equation}
and
\begin{equation}\label{estimate:denominator2}
\Gamma(j)\le\Gamma\bigl(j+\tfrac{s+1}{\ell}\bigr)
\qquad(2\le j,\,0\le s<\ell).
\end{equation}
Now~{\eqref{estimate:numerator}} and~{\eqref{estimate:denominator1}} imply that
\begin{equation}\label{estimate:S1}
S_1(m)\lesssim\frac{\Gamma(n+1)}{\Gamma(n-1)}=n(n-1)
\lesssim 2^n\simeq 2^{m/\ell},
\end{equation}
and, in particular,
\begin{equation}\label{estimate:S2}
S_2(m)
=\sum_{s=0}^r
\frac{\Gamma\bigl(n-1+\frac{r+2}{\ell}\bigr)}
{\Gamma\bigl(n+\frac{r-s+1}{\ell}\bigr)
    \Gamma\bigl(\frac{s+1}{\ell}\bigr)}
\le S_1(m)\lesssim 2^{m/\ell}.
\end{equation}
Moreover, by~{\eqref{estimate:numerator}},
\eqref{estimate:denominator1}
and~{\eqref{estimate:denominator2}} we have that
\begin{eqnarray*}
S_3(m)+S_5(m)
&\lesssim& \sum_{1\le j\le\frac{n}4}
         \frac{\Gamma(n+1)}{\Gamma(j)\Gamma(n-j+1)}
         +\sum_{\frac{3n}4\le j\le n-1}
         \frac{\Gamma(n+1)}{\Gamma(j)\Gamma(n-j+1)}\\
&=& \sum_{1\le j\le\frac{n}4}
\frac{\Gamma(n+1)}{\Gamma(j)\Gamma(n-j+1)}
+\sum_{1\le j\le\frac{n}4}
\frac{\Gamma(n+1)}{\Gamma(n-j)\Gamma(j+1)}\\
&=& \sum_{1\le j\le\frac{n}4}
\frac{\Gamma(n+1)}{\Gamma(j)\Gamma(n-j+1)}\,
\left(1+\frac{n-j}{j}\right)\\
&=&n\sum_{1\le j\le\frac{n}4}
\frac{\Gamma(n+1)}{\Gamma(j+1)\Gamma(n-j+1)}
\le n\sum_{0\le j\le\frac{n}4}\binom{n}{j}.
\end{eqnarray*}
So Chernoff's inequality gives
\begin{equation}\label{estimate:S3-S5}
S_3(m)+S_5(m)\lesssim n\,e^{-n/8}\,2^n\lesssim 2^n
\simeq 2^{m/\ell}.
\end{equation}
To estimate $S_4(m)$ we apply Lemma~{\ref{lem:Stirling}} and
we obtain that
\begin{eqnarray*}
S_4(m)
&\simeq&
\sum_{s=0}^{\ell-1}
\sum_{\frac{n}4<j<\frac{3n}4}
\frac{(n-1)^{\frac{r+2}{\ell}}}
     {j^{\frac{s+1}{\ell}}(n-j)^{\frac{r-s+1}{\ell}}}\,
     \frac{\Gamma(n-1)}{\Gamma(j)\Gamma(n-j)}\\
&\simeq& \sum_{\frac{n}4<j<\frac{3n}4} \frac{\Gamma(n-1)}{\Gamma(j)\Gamma(n-j)}
\simeq \sum_{\frac{n}4<j<\frac{3n}4}
\binom{n}{j}=2^n-2\sum_{0\le j\le\frac{n}4}\binom{n}{j}.
\end{eqnarray*}
Therefore Chernoff's inequality shows that
\begin{equation}\label{estimate:S4}
S_4(m)\simeq 2^n\simeq 2^{m/\ell}.
\end{equation}
By \eqref{estimate:S1}, \eqref{estimate:S2},
 \eqref{estimate:S3-S5} and \eqref{estimate:S4}
 we conclude that $S(m)\simeq 2^{m/\ell}$, and the proof
 is complete.
\end{proof}

\subsection{The small Hankel operator on $L\dla$}\quad\par

In this section we characterize the membership of $\hvla$ to the Hilbert-Schmidt class  $\mathcal{S}_2(L^{2,\ell}_{\alpha})$ of $L^{2,\ell}_{\alpha}$.

Let $L^2_\Delta:=L^2(\C, (1+|z|)^{2(\ell-1)}\,d\nu)$.
Then we have:

\begin{thm}\label{thm:HSL2}
$\mathfrak{h}^{\ell}_{\varphi,\alpha}
\in \mathcal{S}_2(L^{2,\ell}_{\alpha})$ if and only if $\varphi \in
L^2_\Delta$.
Moreover,
\[
\|\mathfrak{h}^{\ell}_{\varphi,\alpha}\|_{\mathcal{S}_2(L^{2,\ell}_{\alpha})}
\simeq\| \varphi\|_{L^2_\Delta}.
\]
 In particular, if
$\varphi \in L^2_\Delta$, then
$\mathfrak{h}^{\ell}_{\varphi,\alpha}\in\mathcal{S}_2(F^{2,\ell}_{\alpha})$ and
$\|\mathfrak{h}^{\ell}_{\varphi,\alpha}\|_
{\mathcal{S}_2(F^{2,\ell}_{\alpha})}\lesssim
\| \varphi\|_{L^2_\Delta}.$
\end{thm}

\begin{proof}
Note that
\[
\mathfrak{h}^\ell_{\varphi,\alpha}(f)(z):=\overline{P\la(\overline{f}
\varphi)(z)}=
\int_{\C}K\la(w,z)f(w)\,\overline{\varphi}(w)\,e^{-\alpha|w|^{2\ell}}\,d\nu(w)
\]
is an integral operator with respect to the positive measure
$e^{-\alpha|w|^{2\ell}}\,d\nu(w)$ and whose integral kernel is
$K^\ell_\alpha(w,z)\overline{\varphi(w)}$. So it is well known
(see~{\cite[Theorem~3.5]{Zhu2007}}, for example) that
\begin{equation*}\begin{split}
\|\mathfrak{h}^{\ell}_{\varphi,\alpha}\|_{\mathcal{S}_2(L^{2,\ell}_{\alpha/2})}^2
&=\int_{\C}\biggl(\int_{\C}
|K^\ell_\alpha(w,z)|^2|\varphi(w)|^2\,e^{-\alpha|w|^{2\ell}}\,d\nu(w)\biggr)
e^{-\alpha|z|^{2\ell}}\,d\nu(z) \\
& =\int_{\C}|\varphi(w)|^2
\biggl(\int_{\C}|K^\ell_\alpha(w,z)|^2e^{-\alpha|z|^{2\ell}}\,d\nu(z)\biggr)
e^{-\alpha|w|^{2\ell}}\,d\nu(w)\\
&=  \int_{\C} |\varphi(w)|^2K\la(w,w)e^{-\alpha|w|^{2\ell}}\,d\nu(w)\\
&\simeq  \int_{\C} |\varphi(w)|^2(1+|w|)^{2(\ell-1)}\,d\nu(w),
\end{split}\end{equation*}
where the last equivalence follows from
$ K\la(w,w)=H\la(|w|^2)\simeq (1+|w|)^{\ell-1}e^{\alpha|w|^{2\ell}}$
 (see \eqref{eqn:KH} and \eqref{eqn:asimpE}).
  And that's all.
\end{proof}

Finally we show that the space of Hilbert-Schmidt symbols for $F\dla$ is just the projection of the space of Hilbert-Schmidt symbols for $L\dla$.

\begin{prop}\label{prop:Ptap}
     The projection $P\la$ is
    bounded from $L^2_\Delta$ onto
    $F^{2,\ell}_{\alpha/2,\Delta}$.
\end{prop}

\begin{proof}
    Let $\{e_m\}_{m\in\N}$ be an orthonormal basis of $F\dla$ and let $\{u_m\}_{m\in\N}$ be an orthonormal basis of the orthogonal of $F\dla$ in $L\dla$.

    By Theorems
 \ref{thm:HSL2} and \ref{thm:main4}, we have, for any $\varphi\in L^2_\Delta$,
\begin{align*}
    \|\varphi\|_{L^2_\Delta}^2\simeq\|\varphi\|_{S_2(L\dla)}^2 &=\sum_{m=1}^\infty \|\hvla(e_m)\|_{L\dla}^2+
    \sum_{m=1}^\infty \|\hvla (u_m)\|_{L\dla}^2\\
    &\ge
    \sum_{m=1}^\infty \|\mathfrak{h}^{\ell}_{P\la(\varphi),\alpha}(e_m)\|_{F\dla}^2\simeq
    \|P\la(\varphi)\|_{F^{2,\ell}_{\alpha/2,\Delta}}^2.
\end{align*}
So we have just proved that $P\la:L^2_\Delta \to
F^{2,\ell}_{\alpha/2,\Delta}$ is bounded.

Let $b\in F^{2,\ell}_{\alpha/2,\Delta}$. Since
$F^{2,\ell}_{\alpha/2,\Delta}\subset F^{2,\ell}_{\alpha/2}$,  we
have $b=P_{\alpha/2}^\ell(b)$, by \eqref{eqn:reproducing:property}.
By \eqref{eqn:kernel}, $K_{\alpha/2}^\ell(z,w)=2^{-1/\ell}
K^\ell_\alpha(z,2^{-1/\ell}w)$, so
\begin{align*}
b(z)&=2^{-1/\ell} \int_\C K^\ell_\alpha\bigl(z,2^{-1/\ell}w\bigr) b(w) e^{-\frac{\alpha}2|w|^{2\ell}}d\nu(w)\\
&=2^{1/\ell} \int_\C K^\ell_\alpha(z,u) b(2^{1/\ell}u)
e^{-2\alpha|u|^{2\ell}}d\nu(u)=P\la(\varphi)(z),
\end{align*}
where
$\varphi(u)=2^{1/\ell} b(2^{1/\ell}u) e^{-\frac{\alpha}4|2^{1/\ell}u|^{2\ell}}$,
which clearly belongs to $L^2_\Delta$.
    \end{proof}


\begin{thebibliography}{12}

\bibitem{Bateman-Erdelyi}
Bateman, H. and  Erd\'elyi, A.
\emph{Higher transcendental functions},
vol. III, McGraw-Hill, New-York-Toronto-London, 1953.

\bibitem{bommier-englis-youssfi}
Bommier-Hato, H., Engli\v{s}, M. and Youssfi, E.H. \emph{Bergman-type projections in generalized Fock spaces},
 J. Math. Anal. Appl. {\bf{389}} (2012), no. 2, 1086–1104.

\bibitem{bommier-youssfi}
Bommier-Hato, H. and Youssfi, E.H.
\emph{Hankel operators on weighted Fock spaces},
Integr. Equ. Oper. Theory {\bf 59} (2007), 1--17.

\bibitem{constantin--ortega-cerda}
Constantin, O. and Ortega-Cerd\`{a}, J.
\emph{Some spectral properties of the canonical solution operator to $\overline{\partial}$ on weighted Fock spaces},
J. Math. Anal. Appl. {\bf 377} (2011), no. 1, 353--361.

\bibitem{constantin-pelaez}
 Constantin, O. and  Pel\'{a}ez, J.A.
\emph{ Boundedness of the Bergman projection on $L^p$-spaces with exponential weights},
  Bull. Sci. Math. {\bf 139} (2015), no. 3, 245--268.


 \bibitem{ConsPelJGA2016}
    Constantin, O. and Pel\'aez, J.A.
    {\em{Integral operators, Embedding theorems and a Littlewood-Paley formula on Fock spaces}},
J. Geom. Anal {\bf{26}} (2016), no.~2, 1109--1154.


\bibitem{dudley}
Dudley, R.M.
\emph{Uniform Central Limit Theorems},
Cambridge Studies in Advanced Mathematics \# 63,
Cambridge University Press, New York, 1999.


 \bibitem{janson-peetre-rochberg}
  Janson, S., Peetre, J. and  Rochberg, R.
  \emph{Hankel forms and the Fock space},
  Rev. Mat. Iberoamericana {\bf 3} (1987), no. 1, 61--138.


\bibitem{marco-massaneda-ortega}
Marco, M., Massaneda, X. and Ortega-Cerd\`{a}, J.
\emph{Interpolating and sampling sequences for entire functions},
Geom. Funct. Anal. {\bf 13} (2003), no. 4, 862--914.

\bibitem{marzo-ortega} Marzo, J. and Ortega-Cerda, J.
 \emph{Pointwise estimates for the Bergman kernel of the weighted Fock space},
 J. Geom. Anal. {\bf 19} (2009), 890--910.

 \bibitem{oliver-pascuas}
 Oliver, R. and Pascuas, D.
\emph{ Toeplitz operators on doubling Fock spaces},
 J. Math. Anal. Appl. {\bf 435} (2016), no. 2, 1426--1457.


\bibitem{OrtegaSeipJAnalMath98}
Ortega-Cerd\`a, J. and Seip, K.
\emph{Beurling-type density theorems for weighted $L^p$ spaces of entire functions},
J. Anal. Math. {\bf 75} (1998), 247--266.

\bibitem{SeiYouJGA2011}
 Seip, K. and Youssfi, E.H.
 {\emph{Hankel operators on Fock spaces and related Bergman kernel estimates}}.
J. Geom. Anal. {\bf{23}} n. 1 (2013), 170--201.

\bibitem{Zhu2007}
Zhu, K.H.
\emph{Operator Theory in Function Spaces},
Second Edition,
Math. Surveys and Monographs \#138,
American Mathematical Society, Providence, Rhode Island, 2007.

 \bibitem{Zhu2012}
Zhu, K.H.
 \emph{Analysis on Fock Spaces},
 Graduate Texts in Mathematics \# 263. Springer, New York, 2012.
\end{thebibliography}
\end{document}